\documentclass[11pt]{amsart}

\usepackage{amssymb, amsthm,amstext,amsfonts,amsmath,graphics,lipsum}

\usepackage[all,2cell]{xy} \UseAllTwocells \SilentMatrices
\usepackage{xypic}

\usepackage{graphicx}

\usepackage[margin=1.2in]{geometry}

\usepackage[latin1]{inputenc}
\usepackage[all]{xy}

\usepackage{tikz}
\usetikzlibrary{matrix}

\makeatletter
\newtheorem*{rep@theorem}{\rep@title}
\newcommand{\newreptheorem}[2]{%
\newenvironment{rep#1}[1]{%
 \def\rep@title{#2 \ref{##1}}%
 \begin{rep@theorem}}%
 {\end{rep@theorem}}}
\makeatother

\newtheorem{theorem}{Theorem}
\newtheorem*{theorem*}{Theorem}
\newreptheorem{theorem}{Theorem}
\numberwithin{theorem}{section}
\newtheorem{theorem-definition}[theorem]{Theorem-Definition}

\newtheorem*{acknowledgements*}{Acknowledgements}

\newtheorem{corollary}[theorem]{Corollary}
\newtheorem*{corollary*}{Corollary}
\newreptheorem{corollary}{Corollary}
\newtheorem{lemma}[theorem]{Lemma}
\newtheorem*{lemma*}{Lemma}

\theoremstyle{definition}
\newtheorem{definition}[theorem]{Definition}
\newtheorem*{definition*}{Definition}
\newtheorem{example}[theorem]{Example}
\newtheorem{claim}[theorem]{Claim}

\newtheorem{note}[theorem]{Remark}
\newtheorem*{note*}{Remark}

\newtheorem{remark}[theorem]{Remark}
\theoremstyle{definition} 
\theoremstyle{remark}
\numberwithin{equation}{section}

\DeclareMathOperator{\Vol}{Vol}

\DeclareMathOperator{\PSL}{PSL}

\DeclareMathOperator{\SL}{SL}

\newcommand{\fun}[3]{#1 \colon #2 \to #3}

\newcommand{\Mod}[1]{\ (\mathrm{mod}\ #1)}

\makeatletter
\newcommand{\tpitchfork}{%
  \vbox{
    \baselineskip\z@skip
    \lineskip-.52ex
    \lineskiplimit\maxdimen
    \m@th
    \ialign{##\crcr\hidewidth\smash{$-$}\hidewidth\crcr$\pitchfork$\crcr}
  }%
}
\makeatother

\begin{document}

\title[Periods of continued fractions and volumes of modular knots complements]{Periods of continued fractions and volumes of modular knots complements }
\author[J. A. Rodriguez-Migueles]{JOSE ANDRES RODRIGUEZ-MIGUELES}
\begin{abstract}
Every oriented closed geodesic on the modular surface has a canonically associated knot in its unit tangent bundle coming from the periodic orbit of the geodesic flow. We study the volume of the associated knot complement with respect to its unique complete hyperbolic metric. We show that there exist sequences of closed geodesics for which this volume is bounded linearly in terms of the period of the geodesic's continued fraction expansion. Consequently, we give a volume's upper bound for some sequences of Lorenz knots complements, linearly in terms of the corresponding braid index. 

Also, for any punctured hyperbolic surface we give volume's bounds for the canonical lift complement relative to some sequences of sets of closed geodesics in terms of the geodesics length.

\end{abstract}
\maketitle
\section{Introduction}

Let $\Sigma$ be a complete, orientable hyperbolic surface or $2$-orbifold of finite area. An oriented closed geodesic $\gamma$ on $\Sigma$ has a canonical lift $\widehat\gamma$ in its unit tangent bundle $T^1\Sigma,$  namely the corresponding periodic orbit of the geodesic flow. Let $M_{\widehat\gamma}$ denote the complement of a regular neighborhood of $\widehat\gamma$ in $T^1\Sigma.$  As a consequence of the Hyperbolization Theorem, $M_{\widehat\gamma}$ admits a finite volume complete hyperbolic metric if and only if $\gamma$ fills $\Sigma$ \cite{FH13}. Such metric is unique up to isometry, by Mostow's Rigidity Theorem, meaning that any geometric invariant is a topological invaraint. Recently, there has been interest in relating the volume of $M_{\widehat\gamma}$ in terms of properties of the closed geodesic $\gamma.$

\vskip .2cm

Bergeron, Pinsky and Silberman have already studied in \cite{BPS16} the problem of finding an upper bound for the volume of $M_{\widehat\gamma}$, by giving one which is linear in the length of the geodesic. Here we improve their upper bound for infinite families of closed geodesics in any punctured hyperbolic surface\string:

\begin{corollary}\label{nub}
Let $\Sigma$ be a punctured surface of genus $g$ with $k$ punctures, admitting a hyperbolic metric $\rho$ and let $d_\Sigma:=6(2g+k-2)$.    Then there exist a constant $C_\rho>0$ and a sequence $\{{\gamma_n}\}$  of filling finite sets of closed geodesics on $\Sigma$ with at most $d_\Sigma$ elements in each set $\gamma_n$ and $\ell_{\rho}({\gamma_n})\nearrow \infty,$  such that  
$$\Vol(M_{\widehat{{\gamma_n}}})\leq 8d_{\Sigma}v_3\left(\frac{C_\rho\ell_\rho({\gamma_n})}{\ln\left(\frac{\ell_\rho({\gamma_n})}{C_\rho}\right)} +2\right),$$
where $v_3$ is the volume of a regular ideal tetrahedron. 
 \end{corollary}

Nevertheless, it is easy to construct sequences of closed geodesics with length approaching to infinity but whose associated canonical lift complements are homeomorphic . For example, the iterations under an infinite-order diffeomorphism of the surface, of a given filling closed geodesic. In (\cite{Rod20}, Theorem 1.1) we constructed more interesting sequences of closed geodesics whose associated canonical lift complements are not homeomorphic with each other and the sequence of the corresponding volumes is bounded. Also in (\cite{Rod20}, Theorem 1.5) we gave a topological lower bound of the volume of $M_{\widehat\gamma}$  in terms of the number homotopy classes of arcs of $\gamma$ in each pair of pants of a given a pants decomposition of $\Sigma.$ This gave us a method to construct sequences of closed geodesics on any hyperbolic surface where the corresponding volumes are bounded from below in terms of the length of the geodesics (\cite{Rod20}, Theorem 1.3). The expresion of this lower bound is similar to the upper bound found in Corollary \ref{nub},  although the sequences of closed geodesics are different. The following result shows that this length bound is sharp for a sequence of filling finite sets of closed geodesics in any punctured hyperbolic surface\string:

\begin{corollary}\label{2}
Let $\Sigma$ be a punctured surface of genus $g$ with $k$ punctures, admitting a hyperbolic metric $\rho$ and let $d_\Sigma:=6(2g+k-2)$.   Then there exist a constant $C_\rho>0$ and a sequence $\{{\gamma_n}\}$  of filling finite sets of closed geodesics on $\Sigma,$ with at most $d_\Sigma$ elements in each set  $\gamma_n$ and $\ell_{\rho}(\gamma_n)\nearrow \infty,$  such that  
	$$\frac{d_{\Sigma}v_3}{12}\left(\frac{\ell_\rho({\gamma_n})}{C_\rho\ln(C_\rho\ell_\rho({\gamma_n}))} -8\right)  \leq \Vol(M_{\widehat{{\gamma_n}}})\leq 8d_{\Sigma}v_3\left(\frac{C_\rho\ell_\rho({\gamma_n})}{\ln\left(\frac{\ell_\rho({\gamma_n})}{C_\rho}\right)} +2\right).$$ 
	where  $v_3$ is the volume of a regular ideal tetrahedron. 
\end{corollary}

 It is interesting to point out that in a collaboration with Tommaso Cremaschi and Andrew Yarmola in  \cite{CRY20} we study the same problem for large families of filling finite sets of simple closed geodesics, such as filling pairs of  simple closed geodesics, and found bounds for the volume of the corresponding link complement in terms of expressions involving distances in the pants graph. As a consequence we constructed sequences of filling pairs of  simple closed geodesics in infinitely many hyperbolic punctured surfaces where the volume of the canonical lift complement is bounded by the logarithm of the length (see \cite{CRY20}, Theorem C). Notice that up to a subsequence in the sequences considered in Corollary  \ref{2}, each element of the sequence is a set of non-simple closed geodesics, so they are not considered in \cite{CRY20}.

\subsection{Modular links and Lorenz links} We focus here mainly in the case of the modular surface $\Sigma_{mod}=\mathbb{H}^2/\PSL_2(\mathbb{Z}).$ This hyperbolic $2$-orbifold is particularly interesting since its unit tangent bundle is homeomorphic to the complement of the trefoil knot in $\mathbb{S}^3.$ Therefore, in this particular case, $M_{\widehat\gamma}$ can be considered as link complement in $\mathbb{S}^3.$ Moreover, after trivially Dehn filling the trefoil cusp of $T^1\Sigma_{mod},$  \cite{Ghy07} Ghys observed that the periodic orbits of the geodesic flow over the modular surface are the Lorenz links in $\mathbb{S}^3.$ 
\vskip .2cm
On the topological side, Lorenz links are prime,  fibered, positive, hence amphicherical, also the link genus and braid index are determined combinatorially (see \cite{BWS83},\cite{Wadd96},\cite{FW87} and \cite{Deh11}). Furthermore, Birman and Kofman proved in \cite{BK09} that Lorenz links and $T$-links coincide. Unfortunately, we do not know if this topological properties are preserved after drilling the trefoil knot, meaning for modular knots in $T^1\Sigma_{mod}.$ However Ghys in \cite{Ghy07} proved that linking between the trefoil knot, cusp of $T^1\Sigma_{mod},$  with the modular knot turns out to be related to the Rademacher function. 
\vskip .2cm
On the geometric side, with respect Thurston`s geometrization of prime knots complements in $\mathbb{S}^3$ \cite{Thu82}, we have that torus knots occur among Lorenz knots. Nevertheless, more than half of the `simplest hyperbolic knots` (whose complements are in the census of hyperbolic manifolds with seven or fewer tetrahedra) are Lorenz knots (see \cite{BK09}, Section 5). Moreover, there are combinatorial upper bounds for the volume of Lorenz links complements (see \cite{CFKNP11}). From the fact that simplicial volume is non-increasing under Dehn filling \cite{Thu79}, here we prove a sharper upper bound for some sequences of Lorenz knots\string:
 
 \begin{corollary}\label{seqL}
 	There exist a sequence $\{K_n\}$ of Lorenz knots  in $\mathbb{S}^3$ such that $n$ is the braid index of $K_n,$ and
 	$$  \Vol(\mathbb{S}^3\setminus K_n)\leq 8v_3(7n+2),$$
 	where $v_3$ is the volume of a regular ideal tetrahedron. If $K_n$ is not hyperbolic, $  \Vol(\mathbb{S}^3\setminus K_n)$ is the sum of the volumes of the hyperbolic pieces of $\mathbb{S}^3\setminus K_n.$
 	\end{corollary}
 
 By comparing this result with the general upper bounds for volumes of Lorenz links complements, obtained by Champanerkar, Futer, Kofman, Neumann and Purcell in \cite{CFKNP11}, we notice that their upper bound (\cite{CFKNP11}, Theorem 1.7) is sharper than ours for the subsequences of Lorenz knots  used in Corollary \ref{seqL} whose braid index is at most $33.$ On the other hand, our upper bound is sharper than theirs (\cite{CFKNP11}, Theorem 1.7) for those Lorenz knots in Corollary \ref{seqL} whose braid index is bigger than $33,$  because their upper bounds are at least quadratically in terms of the braid index and not linear.
\vskip .2cm
Returning to  the modular surface case, in (\cite{BPS16}, Section 3) Bergeron, Pinsky and Silberman also gave an upper bound for the volume of $M_{\widehat\gamma}$ which is proportional to the period of the geodesic's continued fraction expansion plus the sum of the logarithms of its corresponding coefficients. Nevertheless, in (\cite{Rod20}, Corollary 1.2) we constructed sequences of closed geodesics with the period approaching to infinity, but whose sequence of the corresponding volumes is uniformly bounded.  In this paper, we prove that there exist  sequences of closed geodesics for which the volume of the canonical lift complement has an upper bound linearly in terms of the period\string:
\begin{theorem}\label{seq}
For the modular surface $\Sigma_{mod},$ there exist a sequence $\{\gamma_n\}$ of closed geodesics on $\Sigma_{mod}$ such that $n$ is half the period of the continued fraction expansion of $\gamma_n,$ and
$$  \Vol(M_{\widehat{\gamma_n}})\leq 8v_3(7n+2),$$
where $v_3$ is the volume of a regular ideal tetrahedron.
 \end{theorem}
As a consequence of Theorem 1.4 in \cite{Rod20}, we have that up to a constant, Theorem \ref{seq} is sharp. 

\begin{theorem}\label{ub}
For the modular surface $\Sigma_{mod},$ there exist a sequence $\{\gamma_n\}$ of closed geodesics on $\Sigma_{mod}$ such that $n$ is half the period of the continued fraction expansion of $\gamma_n,$ and
$$  v_3 \frac{n}{12} \leq \Vol(M_{\widehat{\gamma_n}})\leq 8v_3(7n+2),$$
where $v_3$ is the volume of a regular ideal tetrahedron. 
 \end{theorem}

\subsection{The thrice-punctured sphere case}

It is interesting to point out that for a thrice-punctured sphere, denoted by $\Sigma_{0,3},$ we proved Corollary \ref{2} for sequences of  closed geodesics (see Corollary \ref{tps}). The main motivation to give this different proof of Corollary \ref{2}  for $\Sigma_{0,3},$ is the next result which estimates a lower bound for the volume of the canonical lift complement of figure-eight type closed geodesics (see Definition \ref{8}) on $\Sigma_{0,3},$ in terms of combinatorial data of the reduced word representing the conjugacy class of the geodesic in $\pi_1(\Sigma_{0,3}).$

\begin{theorem}\label{1}
	
	Given a thrice-punctured sphere $\Sigma_{0,3},$ and $\gamma$ a figure-eight type closed geodesic with respect to $X$ and $Y$ (two free-homotopy classes of distinct punctures in $\Sigma_{0,3}$), we have that\string:
	$$\Vol(M_{\widehat{\gamma}})\geq \frac{v_3}{2}(\sharp\{\mbox{exponents of} \hspace{.2cm}  X\mbox{ in} \hspace{.2cm} \omega_\gamma\}+\sharp\{\mbox{exponents of} \hspace{.2cm}  Y\mbox{ in} \hspace{.2cm} \omega_\gamma\}-2),$$
	where $v_3$ is the volume of the regular ideal tetrahedron, $\omega_\gamma$ is the cyclically reduced word representing the conjugacy class of $\gamma$ in $\langle X,Y\rangle<\pi_1(\Sigma_{0,3}).$ 
	
\end{theorem}

Theorem \ref{1} uses a result due to Agol, Storm and Thurston \cite{AST07} giving a lower bound for the volume of $M_{\widehat\gamma}$ in terms of the simplicial volume of the double of the manifold constructed by cutting $M_{\widehat\gamma}$ along an incompressible surface. Here we apply it to the incompressible surface coming from the pre-image under the map $T^1(\Sigma_{0,3})\rightarrow\Sigma_{0,3}$ of a simple geodesic arc whose end points belong to the same puncture. Moreover, Theorem \ref{1} is an analogue of a lower bound for the volumes of canonical lift complements of geodesics on hyperbolic surfaces admitting a non-trivial pants decomposition (\cite{Rod20}, Theorem 1.5).

\subsection{The sequences of collections of closed geodesics} \label{sequence} Recall that every closed geodesic on the modular surface is represented by a primitive element in the semi-group generated by two parabolic elements   $\footnotesize {X=\begin{pmatrix} 
1 & 1 \\
0 & 1 
\end{pmatrix}}$ 
and $\small {Y=\begin{pmatrix} 
1 & 0 \\
1 & 1 
\end{pmatrix}}$ in $ \PSL_2(\mathbb{Z}) $ (see \cite{BPS17}, Section 3).
\vskip .2cm
These type of closed geodesics can also be found in any hyperbolic surface, and are encoded by primitive elements in the semi-group generated by distinct free homotopy classes of two disjoint and distinct essential simple closed curves denoted by $X$ and $Y.$ In this case we say that the closed geodesics are \textit{figure-eight type} closed geodesics with respect $X$ and $Y.$  The sequences of closed geodesics appearing in this paper are of this kind and can be found  explicitly encoded in Table~\ref{tab}.
\vskip .2cm
Before presenting the structure of this paper, an important question is the distribution of the geodesics obtained here. In the case of the modular surface it will be interesting to calculate  volume bounds for the geodesics which come from the ideal class group of the fields $\mathbb{Q}(\sqrt{d})$ with $d$ a square free positive integer bigger than 1. The interest of these closed geodesics is that they are uniformly distributed on  $T^1\Sigma_{mod}$ (see \cite{Duk88}).  This approach was originally made by Brandts, Pinsky and Silberman in \cite{BPS17} but for only a fixed finite collection of closed geodesics. Unfortunately, we ignore if the sequences of closed geodesics in Theorem \ref{ub} are uniformly distributed. Moreover, in \cite{CKMV21} they give the first known lower bound for the volume of the canonical lift in terms of the length of generic curves in closed surfaces, by applying an exponential multiple mixing result for the geodesic flow. For the modular surface, it will be very interesting to understand the behavior of volumes of canonical lift complements for random closed geodesics in terms of the period of the geodesics continued fraction expansion.
\vskip .2cm
  \paragraph{\textbf{Outline:}} In Section 2 we review the coding of geodesics of the modular surface by positive words in the alphabet $\{X,Y\}$. In Section 3 we review the William's algorithm, giving a combinatorial description of the canonical lift  of figure-eight type closed geodesics. In Section 4 we prove Theorem \ref{seq}, Theorem \ref{ub} and  Corollary \ref{seqL}. In Section 5 we prove  Corollary \ref{nub} and Corollary \ref{2}. And in Section 6 we give a new proof of Theorem \ref{seq} for the thrice-punctured sphere and prove Theorem \ref{1}.

\vskip .2cm
  \paragraph{\textbf{Acknowledgments:}}  I am grateful to Connie On Yu Hui for pointing out a gap, fixing it in the proof of Theorem 1.4 and comments on a previous version of the paper. I thank Sebastian Hensel for explaining the proof of Lemma \ref{cover}. I also thank Pierre Dehornoy, Ilya Kofman,   Fran\c{c}ois Gu\'{e}ritaud and Pekka Pankka for some interesting discussions. I acknowledge the support of the Academy of Finland project 297258 ``Topological Geometric Function Theory" and the Special Priority Programme SPP 2026 Geometry at Infinity funded by the DFG. Moreover, I would like to thank the anonymous referee for many comments and suggestions.

\begin{table}[ht]
\caption{Sequences of filling finite sets of closed geodesics.}
  \centering

    \footnotesize
    \begin{tabular}{|p{2.1cm}|p{2.5cm}|p{.75cm}|p{6.8cm}|p{1.5cm}|} 
    \hline
      \textbf{  $2$-Orbifold $\Sigma$} &\textbf{Geodesics $\{\gamma_n\}$} & \textbf{$\#\gamma_k$  } & \textbf{$\Vol(M_{\widehat{\gamma_n}})$ Bounds} & \textbf{Reference}\\ \hline
      All hyperbolic & All & Any & $\leq C_\rho \ell_\rho(\gamma_k) $
      
      where $C_{\rho}$ is positive constant that only depends on the hyperbolic metric $\rho.$    & \cite{BPS16}, Thm.1.1\\ \hline
      All hyperbolic except spheres with 3 cone points & All & Any & $\geq\frac{v_3}{2} \sum\limits_{P\in \Pi}(\sharp\{\mbox{\footnotesize homotopy Cl. of} \hspace{.1cm}  \gamma_k\mbox{\footnotesize -arcs in} \hspace{.1cm} P\}-3)$ where $\Pi$ is any pants decomposition of $\Sigma$ and $v_3$  the volume of a regular ideal tetrahedron.& \cite{Rod20}, Thm.1.5\\ \hline
     All hyperbolic surfaces & The concatenation of $\alpha_n$ (a fixed closed curve, in a once-punctured torus, with word length at most $n$) and $\eta_0$ (a fixed filling geodesic). 
& 1& $ \geq C_\rho \frac{ \ell_\rho(\gamma_n)}{\ln\ell_\rho(\gamma_n)} $ 

where $C_{\rho}$ is a positive constant that only depend on the hyperbolic metric $\rho$ and $\eta_0.$ &  (\cite{Rod20}, Thm. 1.3)\\ \hline
        $\Sigma_{0,3}$  & $\prod\limits_{i=1}^{n}(X^{k_i}Y^{m_i})$ where   \vskip .2cm
        $k_i,m_i\in \mathbb{N}$& Any & $\geq \frac{v_3}{2}(\sharp\{k_i\}+\sharp\{m_i\}-2)$ & Thm.\ref{1}\\ \hline
      $\Sigma_{mod}$ & $\prod\limits_{i=1}^{n}(X^{k_i}Y)$ 
     and  $\prod\limits_{i=1}^{n}(XY^{m_i})$ with 
       $k_i<k_{i+1}$ and $m_i<m_{i+1}$ &1& $\leq 8v_3(7n+3)$ & Thm.\ref{seq}\\ \hline
      All punctured hyperbolic surfaces & Lifts of $\prod\limits_{i=1}^{n}(X^{k_i}Y)$ and $\prod\limits_{i=1}^{n}(XY^{m_i})$ with
       $3^9<k_i<k_{i+1}$   \vskip .2cm and
       \vskip .2cm
       $3^9<m_i<m_{i+1}$ &  $\leq{d_\Sigma} $ &  $ \leq  8d_{\Sigma}v_3\left(\frac{C_\rho\ell_\rho({\gamma_n})}{\ln\left(\frac{\ell_\rho({\gamma_n})}{C_\rho}\right)} +2\right).$
       
        where $d_\Sigma:=6(2g+n-2)$& Coro.\ref{nub}\\  \hline
       $\Sigma_{mod}$   &  $\prod\limits_{i=1}^{n}(X^{6k_i+1}Y)$ with
       $k_i<k_{i+1}$ & 1& $\asymp n$ 
       
       where $\asymp$ means equality up to constant multiplicative and additive error.& Thm.\ref{ub}\\  \hline
     All punctured hyperbolic surfaces& Lifts of $\prod\limits_{i=3^9}^{n+3^9}(X^{6i+1}Y)$ &  $\leq{d_\Sigma} $& $\asymp d_{\Sigma}v_3\left(\frac{C_\rho\ell_\rho({\gamma_k})}{\ln\left(\frac{\ell_\rho({\gamma_k})}{C_\rho}\right)}\right)$ & Coro.\ref{seq}\\  \hline
       $\Sigma_{0,3}$  & $\prod\limits_{i=1}^{n}(X^{mi+r}Y)$ and $\prod\limits_{i=1}^{n}(X Y^{mi+r})$ with $3^9<m,$ and $0\leq r<m$ & 1& $\asymp v_3\left(\frac{\frac{\ell_{\rho}(\gamma_n)-\delta_{\rho}}{C_{\rho}}}{\ln(C_{\rho}\ell_{\rho}(\gamma_n))}\right)$ 
       
       where $C_{\rho}$ and $\delta_{\rho}$ are positive constant that only depend on the hyperbolic metric $\rho$ and $\gamma_0.$ & Coro.\ref{tps}\\
       \hline  
    \end{tabular}\label{tab}   
\end{table}
\section{Coding of closed geodesics}

\subsection{Modular surface case}\label{ss}
We know that the conjugation classes in $ \PSL_2 (\mathbb{Z}) $ are in correspondence with the closed geodesics on the modular surface. The following  application of the Euclidean Algorithm allows us to code the conjugation classes in $ \PSL_2(\mathbb{Z}) $ in a unique way\string:
\begin{lemma}\label{sta}
Let $A$ be an element of $\SL_2(\mathbb{Z}) $ which has two distinct eigenvalues in $ \mathbb{R}_+ .$ The conjugation class of $ A $ in $ \SL_2 (\mathbb{Z}) $ contains a representative of the form\string:
$$\prod^{n_A}_{i=1}X^{k_i}Y^{m_i} $$
where $X=\begin{pmatrix} 
1 & 1 \\
0 & 1 
\end{pmatrix},$ 
 $Y=\begin{pmatrix} 
1 & 0 \\
1 & 1 
\end{pmatrix}$ and $n_A, k_i, m_i \in\mathbb{N}.$  In addition, the representation is unique up to cyclic permutation of the factors $ X ^ {k_i} Y ^ {m_i}. $ Conversely, any product not empty of such factors is an element of $ \PSL_2 (\mathbb{Z}) $ with two distinct eigenvalues  in $\mathbb{R}^*_+.$
\end{lemma}

Consider the model of the upper half-plane for the hyperbolic space $ \mathbb{H}^2, $ provided with a Farey's triangulation $ \mathcal{F} $ (the ideal triangle with the vertices $ 0,1 $ and $ \infty, $ and all its images by successive reflections with respect to its sides). We know that the group of oriented isometries of $ \mathbb{H}^ 2, $ which preserves $ \mathcal{F}, $ is  identified with $\PSL_2(\mathbb{Z}).$
  \vskip .2cm
Let $ \widetilde \alpha $ be the axis of $A$ oriented towards the attractive fixed point. Then $ \widetilde \alpha $ crosses an infinity of ideal triangles $(...,t_{-1},t_0,t_1,...)$ of $ \mathcal{F}.$ You can formally write a bi-infinite word\string:
$$\omega_A:=...LRRRLLRR...$$
where the $ k^{th}$ letter is $ R $ (resp. $ L $) if and only if the line $ \widetilde \alpha $ comes out from $ t_k $ by the right side (resp. on the left) with respect the side where it enters, in this case we will say that $ \widetilde  \alpha $ turns right (resp. turns left) at $ t_k. $ The word $ \omega_A $ contains at least one $ R $ and one $ L $, because the ends of $ \widetilde  \alpha $ are distinct. The image of $ t_0 $ by $ A $ is a certain $ t_m $ ($ m> 1 $) and $ \omega_A $ is periodical.
  \vskip .2cm
We associate the matrix $ X $ to $ R $ and $ Y $ to $ L, $ which are parabolic transformations of $ \mathbb{H}^2 $ which fix the points $ 0 $ and $ \infty,$ respectively.  Let $ B $ be any subword of $ \omega_A $ of length $ m, $ we look at $ B $ as a product of the matrices $X,Y$ and therefore as an element of $ \SL_2 ( \mathbb{Z}). $ By studying the action of $ X $ and $ Y $ on $ \mathcal{F} $ we can easily see that $ A $ and $ B $ are conjugated in $ \PSL_2 (\mathbb{Z}) $ since both have strictly positive trace.
  \vskip .2cm
We check the uniqueness in the following way: on one hand, if $ A $ and $ B $ are conjugate, there is an element of $ \PSL_2 (\mathbb{Z}) $ (preserving $ \mathcal{F} $) which sends the axis of $ A $ on the axis of $ B, $ therefore $ A $ and $ B $ define the same word $ \omega_A $ up to translation. On the other hand, by considering the action of $ X $ and $ Y $ on $ \mathbb{H}^2, $ we see that a product of the matrices $X,Y$ as in the statement of Lemma \ref{sta} always defines the word $\omega_A= \prod\limits^{n_A}_{i = 1} X^{k_i} Y^{m_i}, $ which repeats infinitely. 
 \begin{definition}\label{perdef}
We denote $n_A$ as the \textit{period} of $\omega_A,$ which is the same as the number of (cyclic) subwords of the form $XY$ in $\omega_A.$
\end{definition}

\subsubsection{The continued fraction expansion of a geodesic in the modular surface}

In this part, we will specify how the sequence $ (k_1, m_1, ..., k_ {n_A}, m_{n_A}) $ of Lemma \ref{sta} is related to the continued fraction expansion of the fixed points of $ A. $ For more details on the proofs of the results on this see  \cite{Dal11}.
  \vskip .2cm
The continued fraction associated with $ x \in \mathbb{R}^*_+ $ can be read in the Farey tiling $ \mathcal{F} $ as follows. We join $ x $ to a point of the imaginary axis of the upper half-plane by a hyperbolic semi-geodesic ray. This arc crosses a sequence of triangles of $ \mathcal{F}. $ We label this arc as before with $ L $ and $ R. $ In the exceptional case where the arc goes through a vertex of the triangle, we choose one or the other of the labels. The resulting sequence $ L^{n_0} R^{n_1} L^{n_2} ... $ with $ n_i \in \mathbb{N} $, is called the cut sequence of $ x . $ If $ x> 1 $ the sequence begins with $ L $, while if $ 0 <x <1 $ the sequence begins with $ R. $ Note that the cutting sequence is independent of the initial point on the imaginary axis. The key observation is\string:

\begin{lemma}
Let $x>1$  with a cutting sequence $L^{n_0}R^{n_1}L^{n_2}...$ with $n_i \in \mathbb{N}.$  Then
$$x=[n_0;n_1,n_2,...].$$
By the same reason $0<x<1$  has a cutting sequence $R^{n_1}L^{n_2}R^{n_3}...$ with $n_i \in \mathbb{N},$ then
$$x=[0;n_1,n_2,...].$$
\end{lemma}

To manage negative numbers, simply replace the negative number $ x $ with $ \frac{-1}{x}= [b_0; b_1, b_2 ,...] $ with $ b_0 \geq 0, $ by using an element of the modular group.

\begin{example} Let $A=\begin{pmatrix} 
1 & 1 \\
-2& 1 
\end{pmatrix}$ the matrix associated with the semi-geodesic ray $ \widetilde \alpha $ which connects the points $ i $ and $ \frac{1 + i}{2}.$ Then $ \widetilde \alpha $ intersects the real axis at $ \frac{\sqrt{5}-1}{2}.$ Following $ \widetilde  \alpha, $ we notice that $ \widetilde  \alpha $ has the cutting sequence $ LRLRLR .. . $. Then $\frac{\sqrt{5}-1}{2}=[0;1,1,1,1,...].$ 
\end{example}

\begin{definition}
We say that two numbers $ x = [a_0; a_1, ...] $ and $ y = [b_0; b_1, ...] $ have the \textit{same tails} if there is $ p, q \in \mathbb{N} $ such that $ a_{p + r} = b_{q + r} $ for all $ r \geq 1 .$ We say that they have the \textit{same tails $ \Mod 2 $} if in addition $ p + q $ is even.\end{definition} 

To understand this definition, we need the following lemma\string:

\begin{lemma}
Let $ \widetilde  \alpha $ and $ \widetilde  \alpha '$ be oriented geodesics lines in $ \mathbb{H}^2 $ with the same positive end point  $ x, $ then the cutting sequences of $ \widetilde  \alpha $ and $ \widetilde  \alpha' $ coincide from a certain rank. In addition, if the end point of $ \widetilde  \alpha $ is $ x = [a_0; a_1,...] ,$ and for $ \widetilde  \alpha '$ is $y= [b_0; b_1,...],$ then $x$ and $y$ have the same tails $ \mod 2 $, if and only if, there is a matrix $ g \in \SL_2 (\mathbb{Z}) $ such that $ g (\widetilde  \alpha ) = \widetilde  \alpha'. $
\end{lemma}

It is not difficult to see algebraically that any number whose continued fraction is almost periodic is quadratic, meaning a solution of an equation of the form\string:
$$ ax^2+bx+c=0\hspace{.2cm}\mbox{with} \hspace{.2cm}  a\in \mathbb{N}^* \hspace{.2cm}\mbox{et} \hspace{.2cm} b,c\in \mathbb{Z}.$$
Conversely, any quadratic number has an almost periodic continued fraction expantion. The following result makes possible to establish a relation between the fixed points of a hyperbolic isometry of $ \SL_2 (\mathbb {Z}) $ and the quadratic numbers\string:

\begin{lemma}
Let $ x $ be irrational. The following properties are equivalent\string:
\begin{enumerate}
\item  [$(i)$]   x is fixed by a hyperbolic isometry of $  \PSL_2(\mathbb{Z})$;
\item  [$(ii)$]  x is quadratic.
\end{enumerate}
\end{lemma}

In conclusion, an element $A$ of $ \PSL_2 (\mathbb{Z}) $ is hyperbolic if and only if it is conjugated in $ \PSL_2 (\mathbb{Z}) $ to an isometry of the form $ \omega_A=\prod\limits^{n}_{i = 1} X^{k_{i}} Y^{m_i}, $ with $ k_i,m_i \in \mathbb{N } $ and its fixed point has the same tail $ \Mod 2 $ as the quadratic number $ x = [0; \overline{k_1, m_1,k_2,m_2, ..., k_n,m_ n}]. $  Notice that the period of the word $\omega_A$ (see Definition \ref{perdef}) is exactly half the period of the continued fraction of $x.$

\subsection{Coding figure-eight type closed geodesics on hyperbolic surfaces} \label{f8g}

\begin{definition}\label{8}
Given a hyperbolic surface $\Sigma,$ we say that a closed geodesic $\gamma$ in $\Sigma$ is a \emph{figure-eight type closed geodesic} if there exist $X$ and $Y$ in $\pi_1(\Sigma)$ representing distinct free homotopy classes of two disjoint  simple closed curves $\alpha$ and $\beta,$ such that $\gamma$ represents the element\string:
$$\omega_\gamma:=\prod\limits_{i=1}^{n_\gamma}X^{k_i }Y^{m_i},$$
 where $k_i ,m_i\in  \mathbb{N},$  $n_\gamma$ is the period of $\gamma$ and $XY$ represents the free homotopy class of a non-simple closed curve.
\end{definition}

 Notice that all figure-eight type closed geodesic relative to $\{X,Y\}$ are contained in a subsurface of $\Sigma$ which is homeomorphic to a thrice-punctured sphere. 

\section{Parametrization of the canonical lift }\label{param}

The path we follow to study the geometry of $ M_{\widehat\gamma} ,$  is by constructing a representant in the same isotopy class as $\widehat\gamma,$ just by using the coding of the word $\omega_\gamma$ (see Lemma \ref{sta}). These representants will be embedded into a particular branched surface. For more details on the algorithm see \cite{BK09}.

 \begin{figure}[h]
 	\centering
 	\includegraphics[scale=0.18] {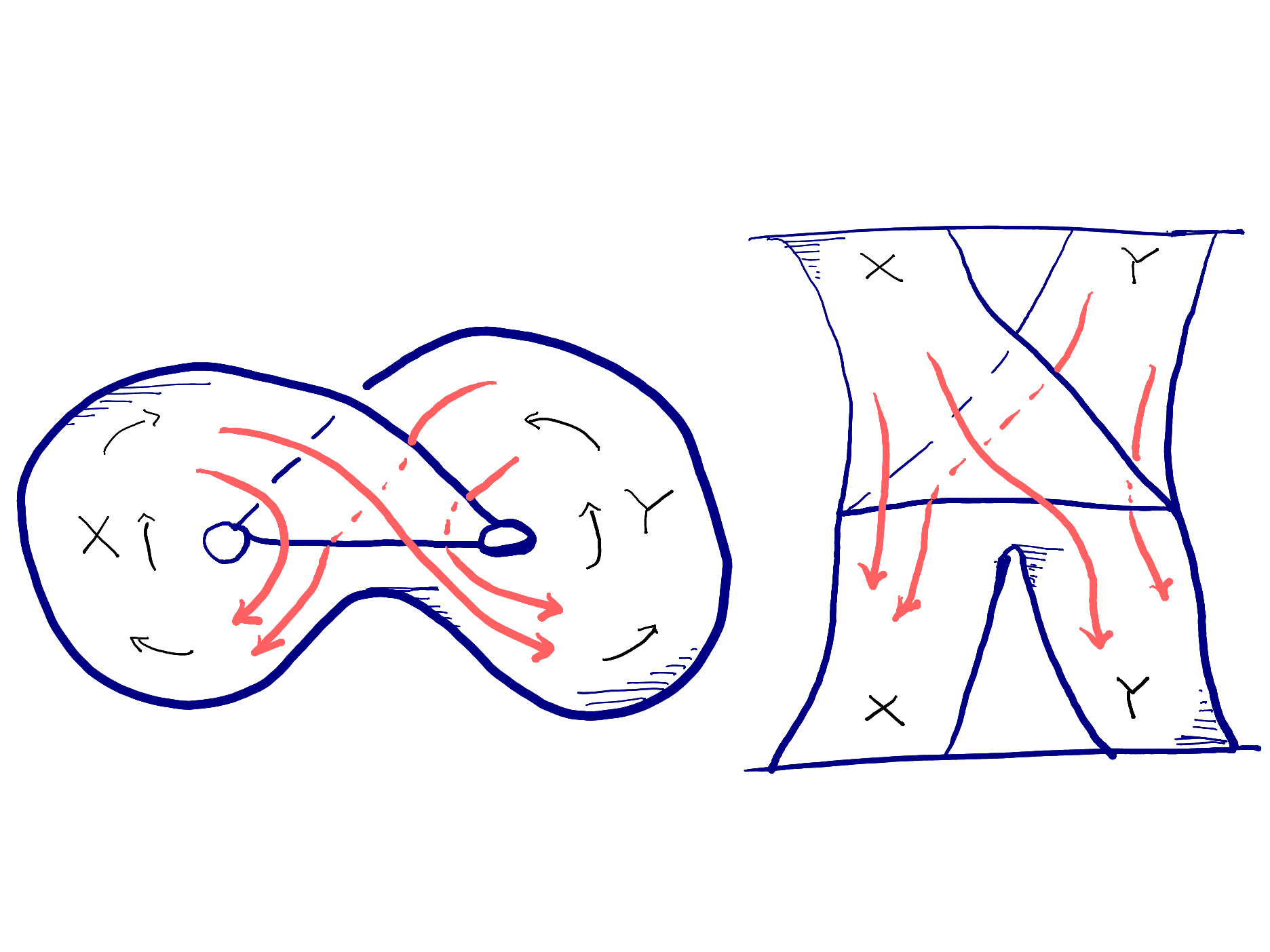}
 	\caption{The figure on the left is the Lorenz template $\mathcal{T}.$ The figure on the right shows the splitting of the template $\mathcal{T}$ to obtain the braid.
 	}\label{lt}
 \end{figure} 
 \vskip .2cm
A \textit{template}  \cite{BWS83} is an embedded branched surface made of several ribbons and equipped
with a semi-flow. A template is characterized by its embedding in the ambient manifold and by the way its ribbons are glued (see Figure \ref{lt}).
   \vskip .2cm
The template for the modular surface case comes with a symbolic dynamics given by the symbols $X$ and $Y$ that correspond to passing through the left or through the right ear (or equivalently through the left or the right half of the branch line). There is not starting point for the orbit, then the words in the alphabet $\{X,Y\}$ used to describe them are primitive  up to cyclic permutation. Ghys proves in \cite{Ghy07}, relying on a theorem of Birman and Williams \cite{BWS83}, that symbolic dynamics of the representation of a periodic orbit in $\mathcal{T}$ is equivalent to its representation as a word in the generator $\{X,Y\}$ for $ \PSL_2 (\mathbb{Z}) .$

 \begin{figure}[h]
\centering
\includegraphics[scale=0.14] {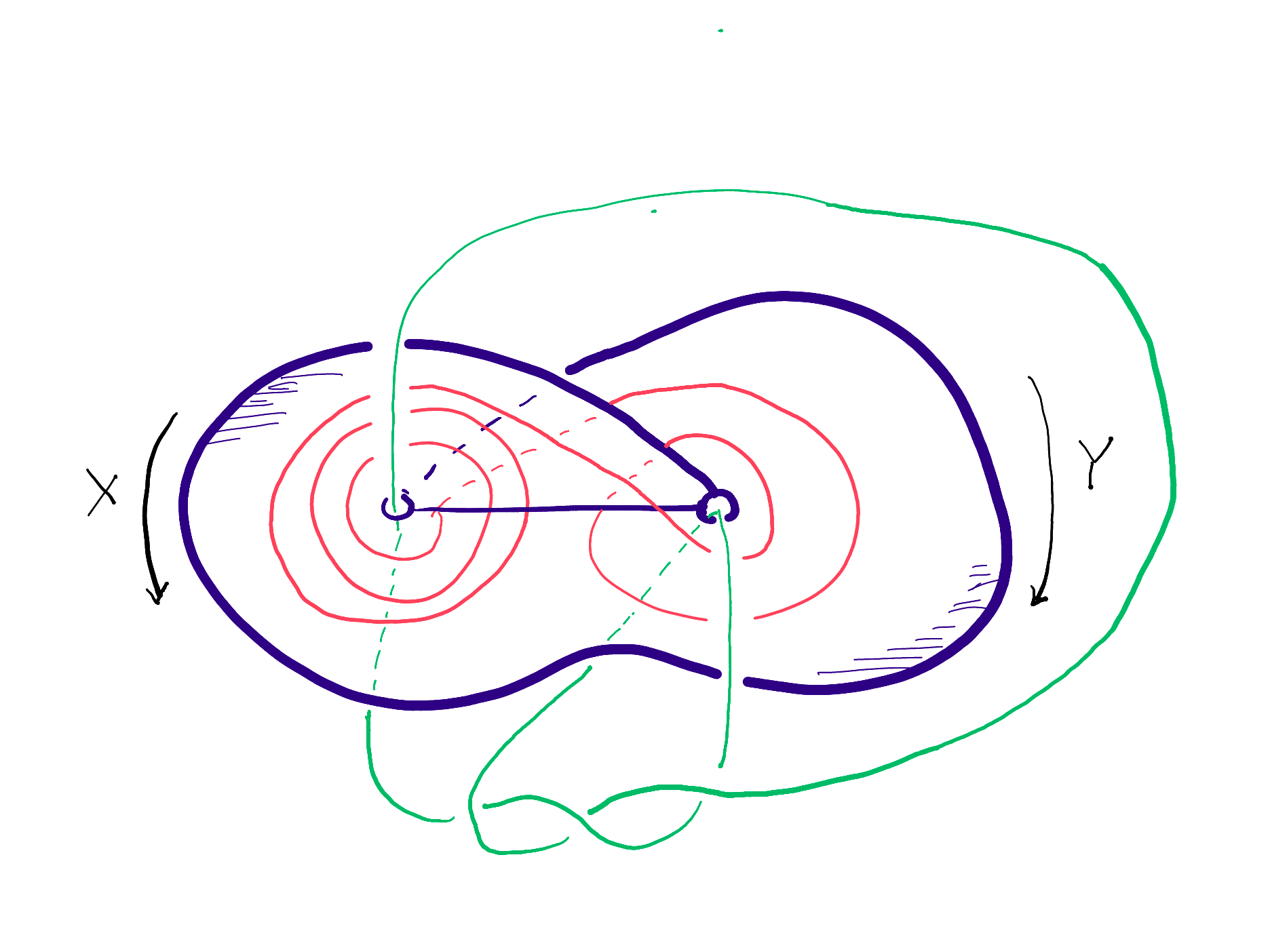}
\caption{The template $\mathcal{T}$ inside the unit tangent bundle of the modular surface (trefoil knot complement in $ \mathbb{S}^3$) with the canonical lift corresponding to the word  $X^3Y^2.$}\label{gl}
\end{figure}

 \begin{theorem}[Ghys]
 	The set of closed geodesics on the modular surface is in bijective correspondence with the set of periodic orbits on the template $\mathcal{T},$ embedded in the trefoil knot complement on the $3$-sphere, excluding the boundary curves of $\mathcal{T}.$ On any finite subset, the correspondence is by an ambient isotopy.
 \end{theorem}
 
Williams in \cite{Will79} constructed an algorithm to find the periodic orbits inside the template $\mathcal{T},$  just from the representing word $\omega.$  In the right of Figure \ref{lt}, the template has been cut open to give a related template for braids, which inherit an orientation from the template, top to bottom. 
  \vskip .2cm
We will illustrate Williams' algorithm in the Figure \ref{lt2}, by showing how to recover the periodic orbit from the word $\omega:=X^4Y^3XY^2.$ We start by  writing the $10$ cyclic permutations   $  \omega= \omega_1, \omega_2,..., \omega_{10},$ in the natural order. We reorder lexicographically these $ 10 $ words using the rule $ X <Y. $ The new position $ \mu_i $ is given after each $ \omega_i: $
\begin{center}

\begin{tabular}{ l | c || r | r}
  $X^4Y^3XY^2$ & 1 & $Y^2XY^2X^4Y$ & 9 \\
  $X^3Y^3XY^2X$ & 2 & $YXY^2X^4Y^2$ & 7 \\
  $X^2Y^3XY^2X^2$ & 3 & $XY^2X^4Y^3$ & 4\\
  $XY^3XY^2X^3$ & 5 & $Y^2X^4Y^3X$ & 8\\
  $Y^3XY^2X^4$ & 10 & $YX^4Y^3XY$ & 6\\
\end{tabular}

\end{center}
This determines a new cyclic order $(1,2,3,5,10,9,7,4,8,6), $ and induces a permutation braid, where the strand $ i $ begins with $ \mu_i $ and ends with $ \mu_ {i + 1}. $  The $i^{th}$ strand of the braid is an \textit {overcrossing strand} if and only if $ \mu_i <\mu_ {i + 1}, $ otherwise it is an \textit {undercrossing strand}. In the example, there are $ 5 $  overcrossing strands and $5$ undercrossing strands, so $ 5 $ strands turn around the left ear and $ 5 $ around the right ear. Beginning with the permutation braid and connecting the end points of the strands with the same index, as in a closed braid, we recover the periodic orbit associated to the cyclic word $ X^4Y^3XY^2 $ (see Figure \ref{lt2}). The braid obtained is called the \textit {Lorenz braid} associated with $ X^4Y^3XY^2. $

 \begin{figure}[h]
\centering
\includegraphics[scale=0.13] {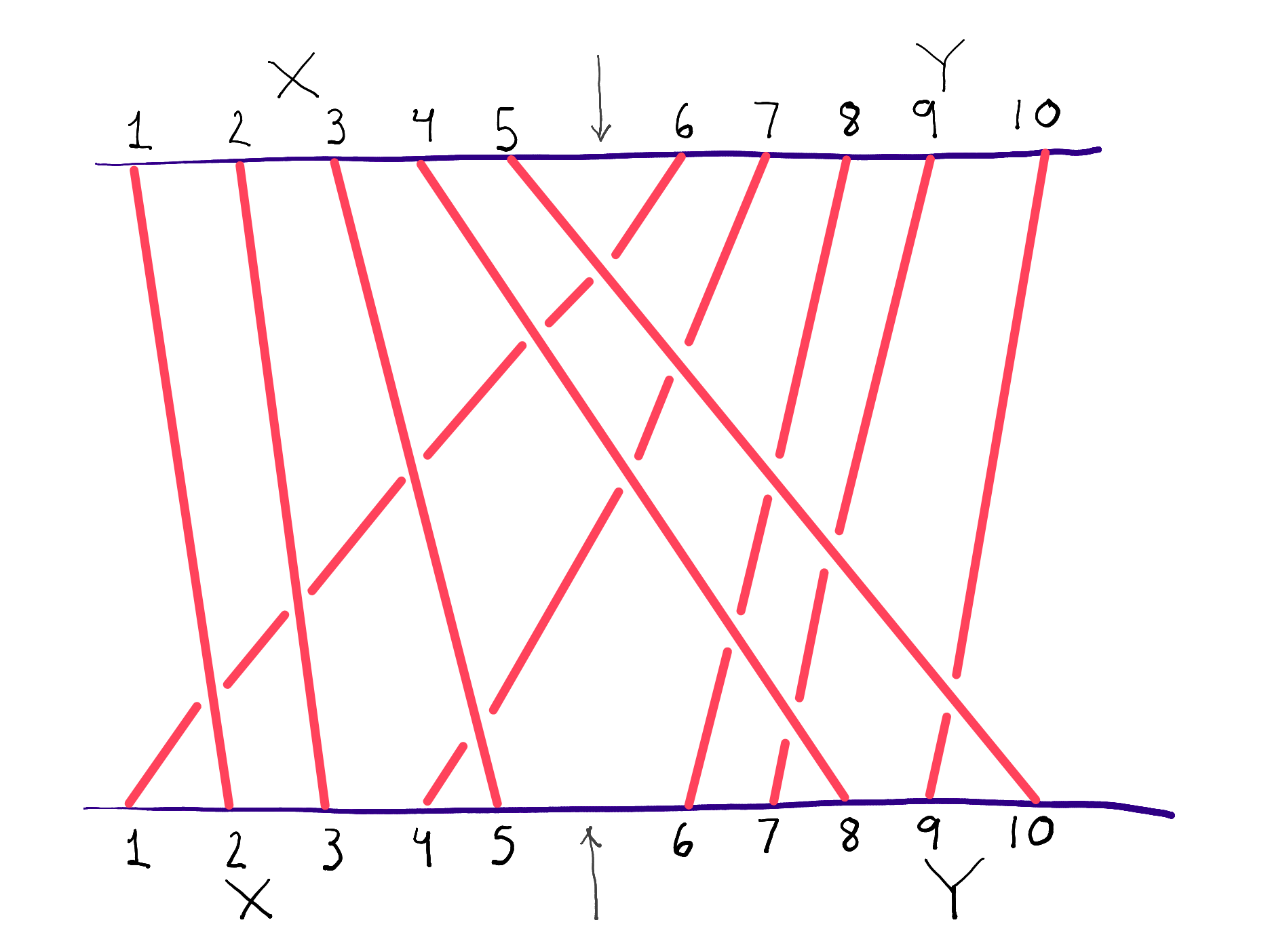}
\caption{The Lorenz braid associatied with $X^4Y^3XY^2$}\label{lt2}
\end{figure}
\vskip .2cm
In a Lorenz braid, two strands of overcrossing (or undercrossing) never intersect, so the permutation associated with the overcrossing stands determine uniquely the rest of the permutation.
  \vskip .2cm
To give a general parameterization of the Lorenz braids, suppose that there are $ p> 1 $  overcrossing strands. On each overcrossing strand, the position of the end will always be greater than that of the initial point. Suppose that the $ i^{th}$ strand begins at $ i $ and ends at $ i + d_i. $ Since two undercrossing strands never cross, we have the following series of positive integers\string:
$$d_1 \leq d_2 \leq...\leq d_{p-1} \leq d_p.$$
We collect this data in the following vector\string:
$$\bar{v} = \langle d_1,...,d_p\rangle_X ,\hspace{.2cm} 1 \leq d_1 \hspace{.2cm}   \mbox{and} \hspace{.2cm}  d_i \leq d_{i+1}.$$

The vector $ \bar{v} $ determines the positions of the strands starting from $ X $ (overcrossing). The strands starting from $ Y $ (undercrossing) fill the remaining positions, so that all the crossings are formed between the overcrossing strands and the undercrossing strands. In Figure \ref{lt2}, the middle arrows separate the left and right strands. Each $ d_i $ with $ i = 1, ..., p $ is the difference between the initial and final positions of the $ i^{th}  $ overcrossing strand. The integer $ d_i $ is also the number of strands that pass under the $ i^{th}$ braid strand. The vector $ \bar{v} $ determines the Lorenz braid with $ n = (p + d_p) $ strands.  All periodic orbits on the template $\mathcal{T}$   appear in this way.
\vskip .2cm
The overcrossing strands travel in groups of \textit{parallel} strands, which are strands of the same slope, or equivalent strands whose associated $ d_i $ coincide. If $ d_{\mu_j} = d _ {\mu_j + 1} = ... = d _ {\mu_j + s_j + 1}, $ where $ s_j $ is the number of strands in the $ j^{th}$ group then let $ r_j = d_{\mu_j} $. Thus, we can write $ \bar{v} $ in the form\string:
$$\bar{v} =\langle d^{s_1}_{\mu_1 } ,...,d^{s_k}_{\mu_k}  \rangle_X =\langle r^{s_1}_1,...,r^{s_k}_k\rangle_X , \hspace{.2cm}1\leq s_i\hspace{.2cm}   \mbox{and} \hspace{.2cm}  r_i<r_{i+1}.$$
Note that
$$p=s_1 +...+s_k,\hspace{.2cm}  d_1 =r_1, \hspace{.2cm} d_p =r_k.$$ 
The period of the word $\omega$ is found in terms of the braid representation by the following number\string:
$$t=\sharp\{i \hspace{.2cm} | \hspace{.2cm} i+d_i >p \hspace{.2cm}   \mbox{where} \hspace{.2cm}1\leq  i\leq  p\}.$$
It is known that the braid index of a Lorenz knot is the period of $\omega,$ a concept that was first encountered in the
study of Lorenz knots from the point of view of symbolic dynamics (see \cite{BWS83}).
\vskip .2cm
In the  example of Figure \ref{lt2}, $ \langle 1,1,2,4,5 \rangle_X = \langle 1^2,2^1,4 ^1,5^1 \rangle_X. $ Also, $ p = 5, $ $ k = 4, $ $ r_k = 5, $  $ n = p + r_k = 10 $ and the braid index is $ 2 .$
\vskip .2cm
In the following result we show explicitly the coefficients of the Lorenz braid vector for infinitely many words in $\{X,Y\}.$ This will be used in almost every result in this paper.
\begin{lemma}\label{para}
Let $(k_i)_{i=1}^{n}\in \mathbb{N}^n$ such that $k_1+1<k_2,$  $k_i<k_{i+1}$ for $2\leq i\leq n-1,$ and $\omega=\prod\limits_{i=1}^{n}(X^{k_{n+1-i}}Y).$  Then the associated Lorenz braid to $ \omega $  is\string:
$$ \langle 1^{s_1},2^{s_2},...,(n-1)^{s_{n-1}},n^{s_n}\rangle_X ,$$
where $ s_i=i(k_{n+1-i}-k_{n-i})$ for $1\leq i\leq n-2,$  $ s_{n-1}=(n-1)(k_2-k_1-1),$ and
$$ s_n =n(k_1+1)-1.$$
\end{lemma}
\begin{proof}[\bf{Proof:}]
By induction over $n.$ For the base of induction consider $n=3,$ it determines the following permutation, obtained by reordering lexicographically the $k_3+k_2+k_1+3$ words induced by $\omega_\gamma$ under cyclic permutation\string:

{\tiny $$(1,...,k_3-k_2,(k_3-k_2)+1+2j,(k_3-k_2)+2(k_2-k_1)+2+3l, (k_3-k_2)+2+2j, (k_3-k_2)+2(k_2-k_1)+3+3l, (k_3-k_2)+2(k_2-k_1)+1+3l),$$}
where $0\leq j\leq k_2-k_1-1,$ and $0\leq l\leq k_1.$
\begin{enumerate}
    	\item The number of overcrossing strands shifted one place are the first $k_3-k_2$ then $s_1=k_3-k_2.$
    	\item The number of overcrossing strands shifted two places are the ones whose ends have a $j>0$ index then $s_2=2(k_2-k_1-1).$
    	\item The number of overcrossing strands shifted three places are the ones whose ends has a $l$ index, with the exception of $(k_3-k_2)+2(k_2-k_1)+1$ which correspond to an undercrossing strand, then $s_3=3(k_1+1)-1.$
    	\end{enumerate}
If our statement is true for $n=m,$ then after multiplying $\prod\limits_{i=1}^{m}(X^{k_i}Y)$ with $X^{k_{m+1}}Y$ we will modify the Lorenz braid by adding $k_{m+1}$ overcrossing strands\string:
\begin{enumerate}
    	\item $k_{m+1}-k_m$ at the begining of the braid, then $s_{m+1}=k_{m+1}-k_m.$
    	\item $k_{m+1-i}-k_{m-i}$ are added in the $i^{th}$ collection of parallel strand for $1\leq i\leq m-2,$ because one new overcrossing strand enters each collection of parallel strands of the previouse Lorenz braid. This implies that\string: $$s_{i+1}=i(k_{m+1-i}-k_{m-i})+k_{m+1-i}-k_{m-i}=(i+1)(k_{m+1-i}-k_{m-i}).$$
    	\item For the penultimate and last collection of parallel overcrossing strands we will add $k_2-k_1-1$ and $k_1+1$ respectively because of the new strand entering into each collection of parallel overcrossing strands. This implies that\string: $$s_m=(m-1)(k_2-k_1-1) +k_2-k_1-1\hspace{.1cm}   \mbox{and} \hspace{.1cm}s_{m+1}=(m+1)(k_1+1)-1.$$
    	\end{enumerate}
\end{proof}
\begin{note}
Notice that by turning over the Lorenz template we have the analogue result to Lemma \ref{para} for the words $\prod\limits_{i=1}^{n}(XY^{m_i}),$  such that $m_1+1<m_2,$ and $m_i<m_{i+1}$ for $2\leq i\leq n-1.$ 
\end{note}

 \subsection{Canonical lifts of figure-eight type closed geodesics}
 
Given $\gamma$ a figure-eight type closed geodesic on a hyperbolic surface, we can restrict to the pair of pants $\Sigma_{0,3}$ where $\gamma$ is filling. We are going to construct an explicit representant of the isotopy class of $\widehat\gamma$ in $T^1(\Sigma_{0,3}).$ 

\begin{lemma}
The set of canonical lifts relative to figure-eight type closed geodesics on $\Sigma_{0,3}$ with respect two different fixed punctures, is in bijective correspondence with the set of periodic orbits on the template $\mathcal{T}.$
\end{lemma}
\begin{proof}[\bf{Proof}]
Let $\gamma$ be a figure-eight type closed geodesic on $\Sigma_{0,3}$ with respect two different fixed punctures. 
\vskip .2cm
First we will fix a vector field on $\Sigma_{0,3},$ to do so consider the corresponding Lorenz braid in $\mathcal{T},$ induced by the word $\omega_\gamma.$ Notice that there is a natural projection of $\mathcal{T}$  to a pair of pants $\Sigma_{0,3}$ which is obtained by the overlaping the ears of the template $\mathcal{T}$ (see Figure \ref{pro}).
 \begin{figure}[h]
\centering
\includegraphics[scale=0.17] {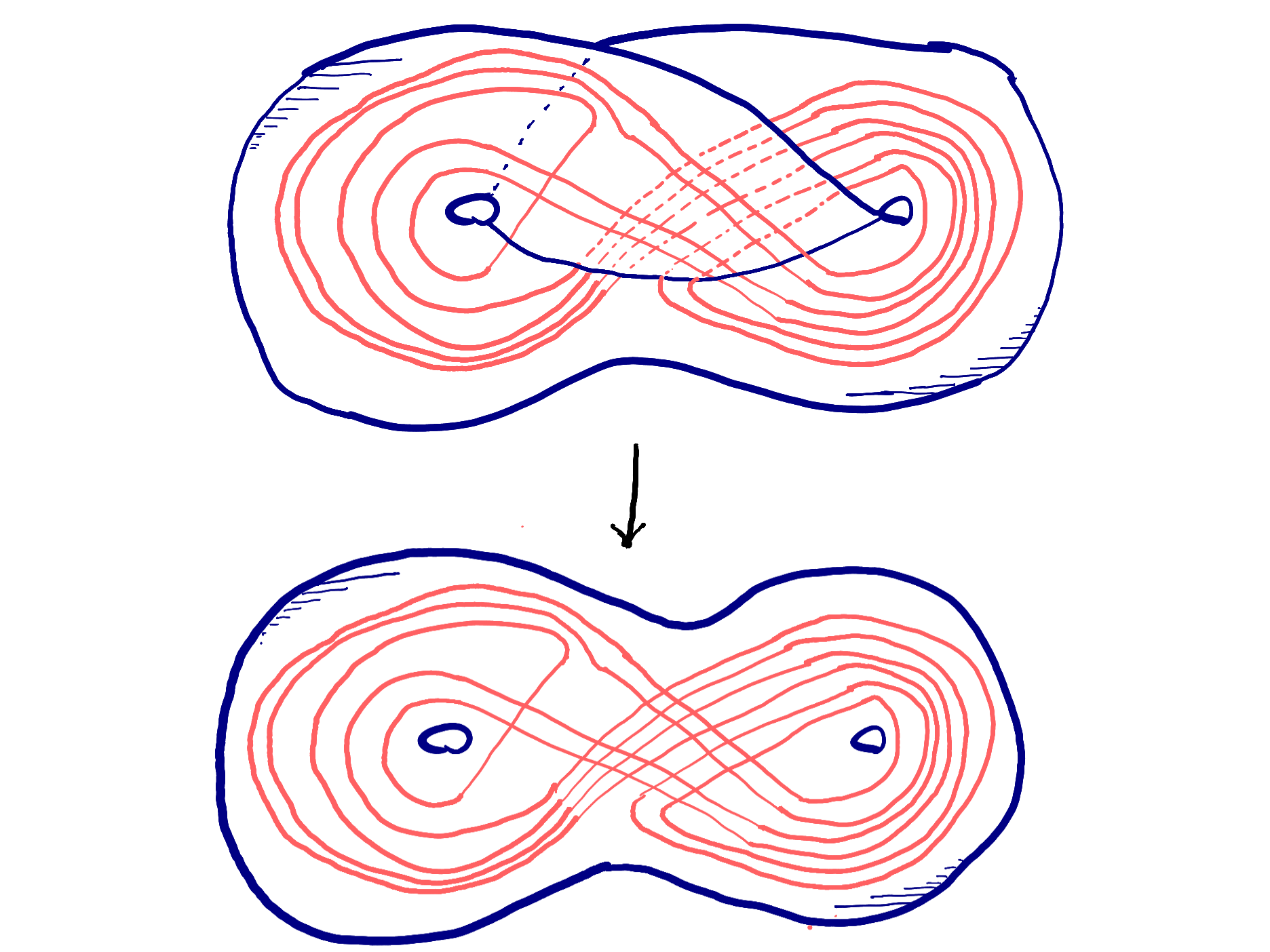}
\caption{The projection map from $\mathcal{T}$  to $\Sigma_{0,3}$}\label{pro}
\end{figure}

\vskip .2cm
Moreover, the periodic orbit is mapped under this projection to a figure-eight type closed curve in minimal position and representing the same homotopy class as $\gamma$ in $\Sigma_{0,3}.$ 
\vskip .2cm
The vector field $\xi$ is obtained by extending the the oriented foliation given by the oriented arcs of $\gamma$ in $\Sigma_{0,3}$ outside the overlapping triangle of the projection map from the template $\mathcal{T}$  to $\Sigma_{0,3}.$ The oriented foliation is obtain first by enclosing each arc to a family of disjoint concentric circles to the corresponding puncture, and then doing parallel copies of the closed simple loops which converges from each side to the boundary of the piece of $\Sigma_{0,3}$ obtained after splitting along a mid edge connecting the third boundary component of $\Sigma_{0,3}$ (see Figure \ref{fie}).

 \begin{figure}[h]
\centering
\includegraphics[scale=0.17] {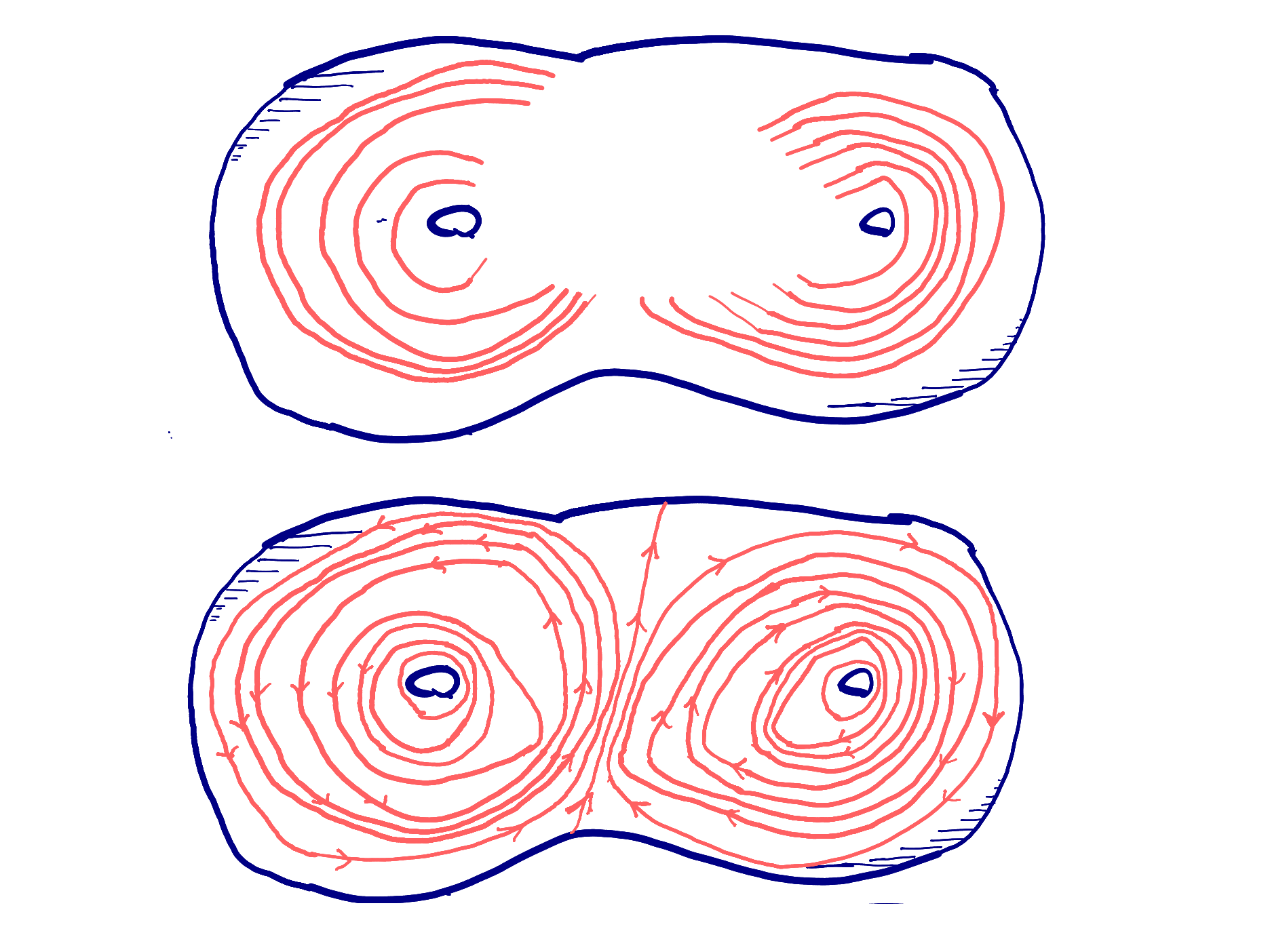}
\caption{The oriented foliation induced by the $\gamma$-arcs outside the non-injective piece of the projection map.}\label{fie}
\end{figure}
\vskip .2cm
The vector field $\xi$ induces a global section inside $T^1(\Sigma_{0,3})$ and the canonical lift $\widehat \gamma$ is isotopic to the embedding of the Lorenz braid in $\mathcal{T}$ associeted to the word $\omega_\gamma.$ 
\end{proof}

\section{Sequences of geodesics on the modular surface whose canonical lift complement volume is bounded linearly by the period}\label{modular1}

In this Section we prove, in Theorems \ref{seq} and  \ref{ub}, a bound for the volume of the complement of canonical lifts for an infinite family of geodesics in the modular surface and such bound depends linearly on the period of the geodesic's continued fraction expansion. Also we give in Corollary \ref{seqL} an upper bound for volumes of some sequences of Lorenz knots complements.
\vskip .2cm
To get to imagine the canonical lifts in Theorems \ref{seq} and  \ref{ub}, we will exemplify the results by the canonical lift associated to the closed geodesic $X^{11}YX^{10}YX^8YX^5YXY$ on the modular surface. As we saw in Section \ref{param} we can also associate it with the Lorenz braid $ \langle 1^1,2^4,3^9,4^{12},5^9\rangle_X $ (see Figure \ref{code0}). In order to simplify the figures in this Section, we will remove the trefoil component from the link associated to the cusp of $ T^1\Sigma_{mod}$ (see Figure \ref{gl}), and just focus on the Lorenz braid. 

\begin{figure}[h]
	\begin{center} 
		\includegraphics[scale=0.75] {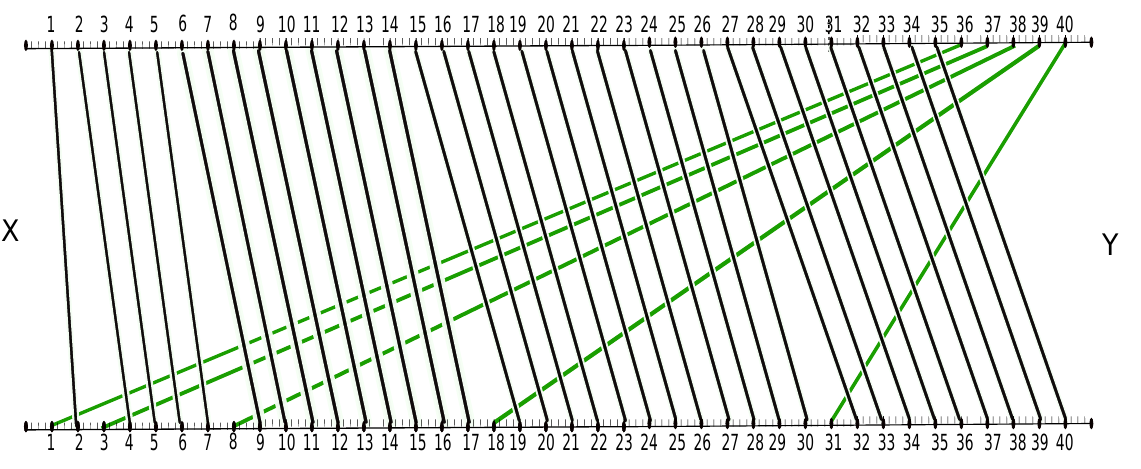}
		\caption{$X^{11}YX^{10}YX^8YX^5YXY.$}\label{code0}
	\end{center} 
\end{figure}

\begin{reptheorem}{seq}
	For the modular surface $\Sigma_{mod},$ there exist a sequence $\{\gamma_n\}$ of closed geodesics on $\Sigma_{mod}$ such that $n$ is half the period of the continued fraction expansion of $\gamma_n,$ and
	$$  \Vol(M_{\widehat{\gamma_n}})< 8v_3(7n+2),$$
	where $v_3$ is the volume of a regular ideal tetrahedron.
\end{reptheorem}

We will use the fact \cite{Thu79} that if a compact orientable hyperbolic $3$-manifold $M$ is obtained by Dehn filling another hyperbolic $3$-manifold $N,$ then the volume of $M$ is less than the volume of $N.$ So the two key ideas to give the upper bound are\string:
\begin{enumerate}
    \item Construct a link $ L_\gamma $ associated with $ \gamma $ in $ \mathbb{S}^3 $ such that by Dehn  filling  along some components of $ L_\gamma, $ we get $M_{\widehat{\gamma}}.  $
    \item Notice that $\mathbb{S}^3\setminus L_\gamma $ is homeomorphic to a  link (with three components) complement on a circle bundle over a punctured sphere, where projection of the three knots  to the base punctured sphere are three closed curves. Then by (\cite{CR18}, Thm 1.5) the volume can be bounded from above by the self-intersection number of the three closed curves.
\end{enumerate}
 
Before stating the proof of Theorem \ref{seq}, we recall some tools that will be used in the proof of this result.

\subsection{Vertical annuli}

Here we will construct some parallel annuli around the ribbons of the Lorenz template, that will help us to reduce the complexity by applying Dehn filling surgery on an augmented Lorenz braid $ \langle 1^{s_1},2^{s_2},...,(n-1)^{s_{n-1}},n^{s_n}\rangle_X $ where $ s_i=i(k_{n+1-i}-k_{n-i})$ for $1\leq i\leq n-2,$  $ s_{n-1}=(n-1)(k_2-k_1-1),$ and $ s_n =n(k_1+1)-1 $.
\vskip .2cm
We define the following $ n $ vertical annuli $ A_{X_i} $ delimited by unknots parallel to the $ X$-band, which enclose the intervals\string:
\begin{enumerate} 

    \item for $ i < n-1 ,$ let $s_0=0$ and we have   $$\left(\sum^{i-1}_{k=0}s_k+3/4, \sum^{i}_{k=0}s_k+1/4\right),$$

    \item for $ i = n-1 $ we have 
$$\left(\sum^{n-2}_{k=0}s_k+3/4,\sum^{n-1}_{k=0}(n-1)+s_k+1/4\right),$$

\item for $ i = n $ we have
$$\left(\sum^{n-1}_{k=0}(n-1)+s_k+3/4,\sum^{n}_{k=0}s_k+1/4\right),$$
\end{enumerate}
We finally add a last annulus $ A_{Y_1} $ delimited by unknots parallel to the $ Y$-band, which enclose the interval\string:
$$\left(\sum^{n}_{k=0}s_k+3/4,\sum^{n}_{k=0}s_k+n+1/4\right).$$

\begin{figure}[h]
	\begin{center} 
		\includegraphics[scale=0.75] {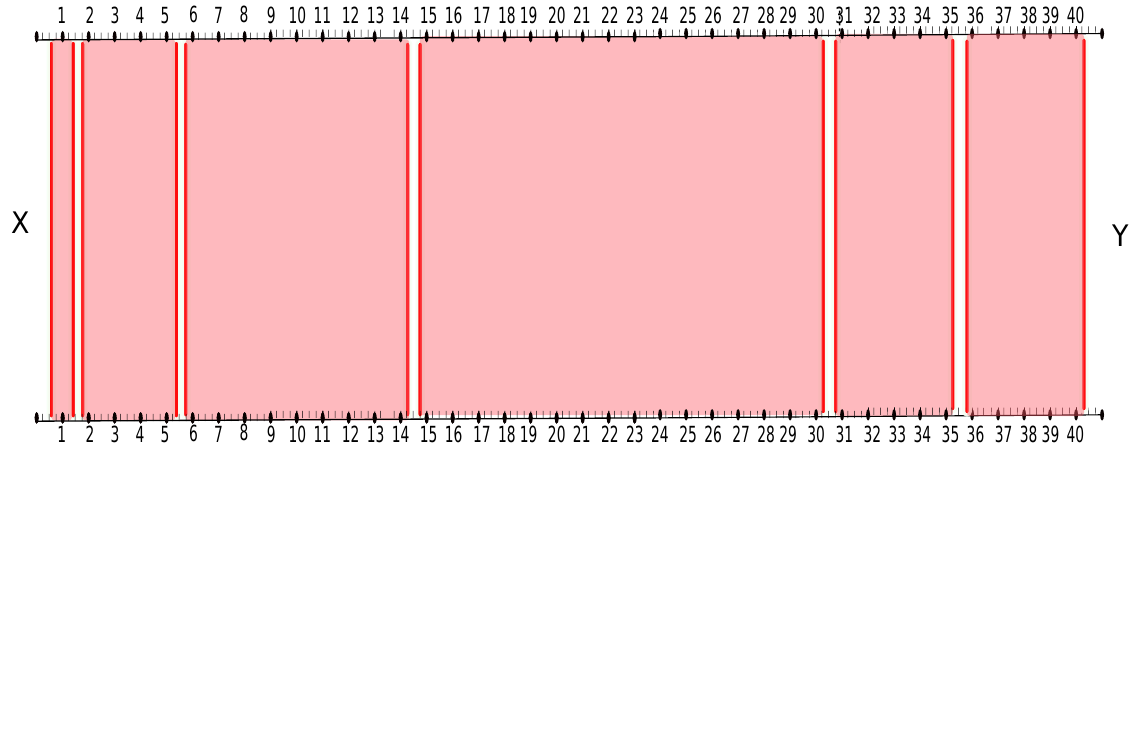}
		\caption{The vertical annuli of $X^{11}YX^{10}YX^8YX^5YXY$ (top). And the same vertical annuli embedded in the Lorenz template $\mathcal{T}$ (bottom left) and the splitting of the template $\mathcal{T}$ (bottom right).}\label{code1a}
	\end{center} 
\end{figure}

In the following Subsection we explain the Dehn surgery that we will apply along the boundary components of the vertical annuli, to relate $\widehat{\gamma}$ with a simplified knot.

\subsection{Annular Dehn surgery} 

Given  $ A $ an  annulus embedded in the interior of an oriented $ 3 $-manifold $ N,$  let $ \partial A=L_{+ 1} \sqcup L_{- 1}. $ Let us orient $ L_{+ 1} $ and $ L_{ -1} $ with a compatible orientation of $N$. Let $ \{m_i, l_i \}, $ $ i = \pm 1, $ be a basis where $ m_i $ is the meridian over $ \partial N(L_i) $ and $ l_i $ a longitude over $ \partial N(L_i) $ induced by $ A $ for $ i = \pm1$ (i.e. $ l_i = \partial N (L_i) \cap A). $
\vskip .2cm
With these bases, the Dehn filling along the slope $ 1/n $ in $ L_{+ 1} $ and $-1/n$ in $ L_{- 1} $ gives a homeomorphism $ N \cong N_{(L_ + \cup L_-)} \left(\frac{1}{n}, \frac{-1}{n}\right). $ This homeomorphism is obtained by cutting $ N \cong N _ {(L_ + \cup L _-)} $ along  $ A $ then wind $n$-times, and nontrivially filling the boundary components of $A$. For example, see Figure \ref{adt} for the case $n=-2 $ and its effect on the curve that passes through $A.$ Notice that this homeomorphism is the identity outside a normal neighborhood of $ A. $

\begin{figure}[h]
	\centering
	\includegraphics[scale=0.3] {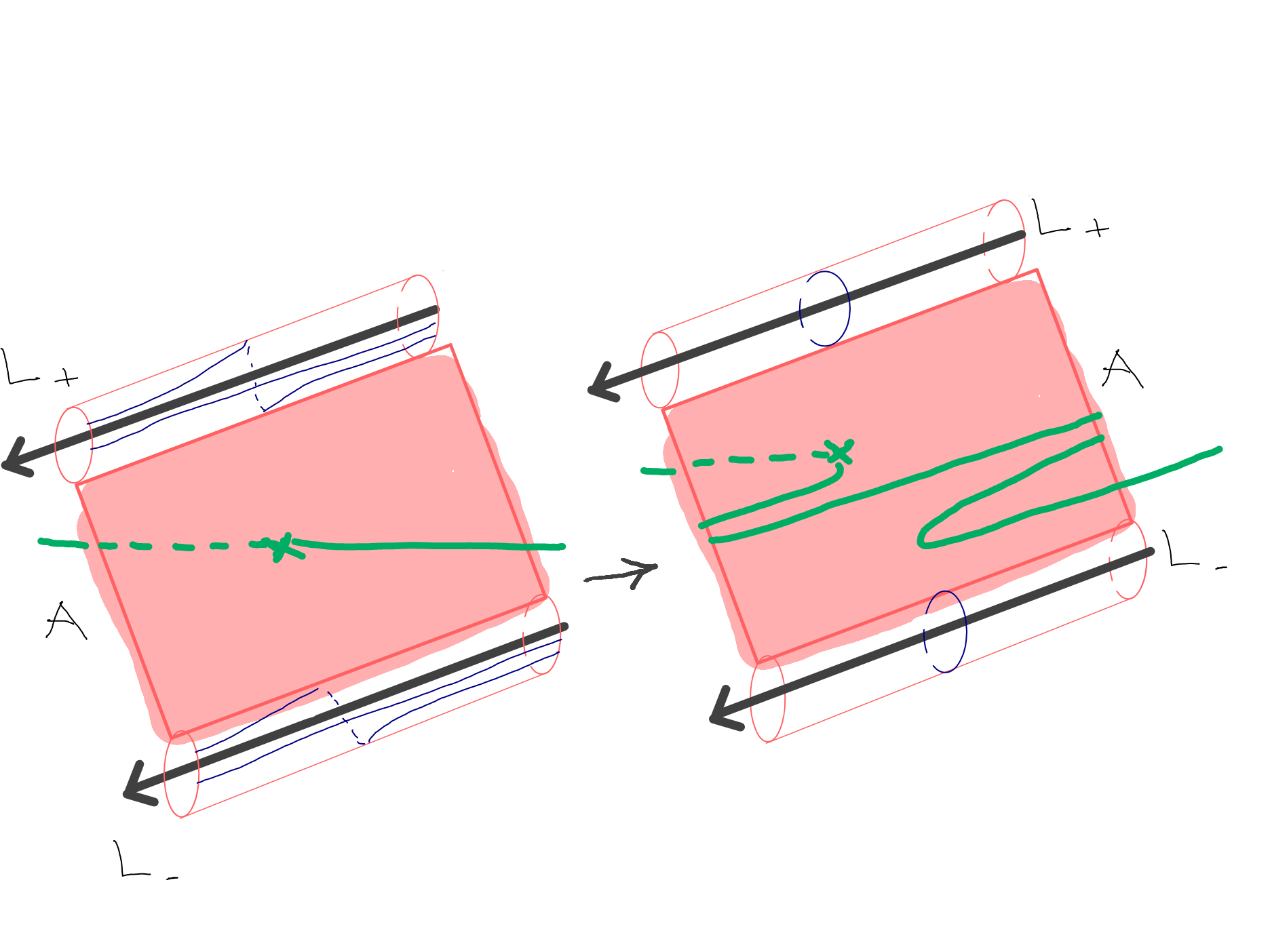}
	\caption{The twice annular Dehn surgery surgery makes the green curve to twist along $A$ two times.
	}\label{adt}
\end{figure}

If an oriented surface $S$ passes through the annulus $ A $ such that the intersection of $S$ and $A$ is a simple closed curve $c$ which is essential in $ A, $ then the homeomorphism
$$ N\rightarrow N_{(L_+ \cup L_-)} \left(\frac{1}{n}, \frac{-1}{n}\right) $$
restricted to $ S $ is a $n$-Dehn twist. If $ n> 0 $ and if we look $ S $ on the same side as $ L_+, $ then the Dehn twist is on the opposite sense.
\vskip .2cm
We now prove the upper bound of Theorem \ref{seq}\string:

\begin{proof}[\bf{Proof of Theorem \ref{seq}:}]
		Let $(k_i)_{i=1}^{n}\in \mathbb{N}^n$ such that $k_1+1<k_2,$  $k_i<k_{i+1}$ for $2\leq i\leq n-1,$  and $\gamma$ the closed geodesic on $\Sigma_{mod}$ associated to $\prod\limits_{i=1}^{n}(X^{k_{n+1-i}}Y),$ where $\footnotesize {X=\begin{pmatrix} 
			1 & 1 \\
			0 & 1 
			\end{pmatrix}}$ 
		and $\small {Y=\begin{pmatrix} 
			1 & 0 \\
			1 & 1 
			\end{pmatrix}}$ in $ \PSL_2(\mathbb{Z}) $ (see Lemma \ref{sta}) .
		\vskip .2cm
		By Lemma \ref{para} the associated Lorenz braid of $ \gamma $  is\string:
		$$ \langle 1^{s_1},2^{s_2},...,(n-1)^{s_{n-1}},n^{s_n}\rangle_X ,$$
		where $ s_i=i(k_{n+1-i}-k_{n-i})$ for $1\leq i\leq n-2,$  $ s_{n-1}=(n-1)(k_2-k_1-1),$ and
		$ s_n =n(k_1+1)-1.$

		\begin{figure}[h]
			\begin{center} 
				\includegraphics[scale=0.45] {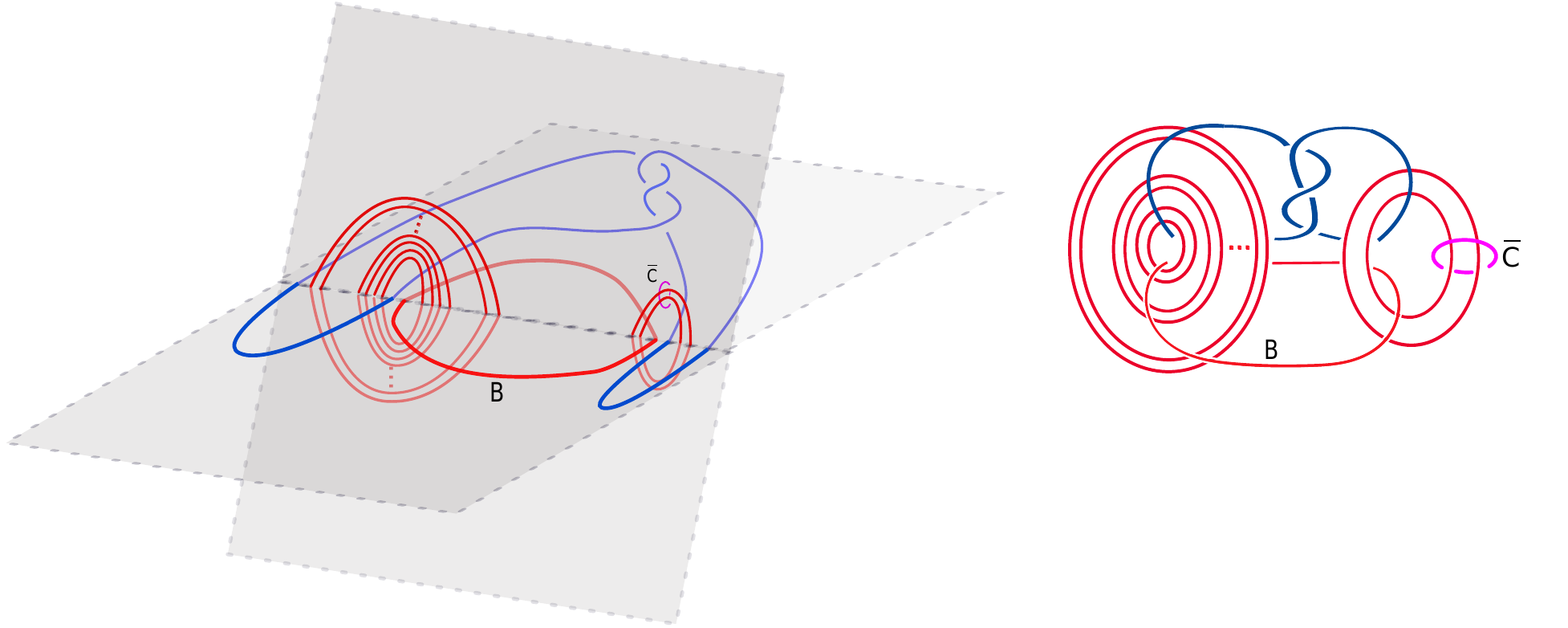}
				\caption{The link $\bar\tau\cup B\cup \bar{C}\cup\{\partial_L A_{X_i},\partial_R A_{X_i}\}_{i=1}^{n}\cup\{\partial_L A_{Y_1},\partial_R A_{Y_1}\}$ }\label{pun1}
			\end{center} 
		\end{figure}
		
	First we construct the link $L_\gamma$ which will consist of $2n+5$ components, where $2n+6$ of them are\string:
		\begin{enumerate}
			\item The trefoil knot corresponding to the boundary of $T^1\Sigma_{mod},$ denoted by $\bar\tau,$
			\item the boundary components of each vertical annulus $A_{X_i} $ and $A_{Y_1}$ in $T^1\Sigma_{mod},$ denoted by $\{\partial_L A_{X_i},\partial_R A_{X_i}\}_{i=1}^{n}$ and $\{\partial_L A_{Y_1},\partial_R A_{Y_1}\},$
			\item an unknot $B$ enclosing only the links in $(2).$ 
\item an unknot $\bar{C}$ enclosing only $\{\partial_L A_{Y_1},\partial_R A_{Y_1}\},$ (see Figure \ref{pun1}).
   
		\end{enumerate}

			Before constructing the last component of $L_\gamma$ notice that the link complement formed by $(1),(2),(3)$ and $(4)$ is homeomorphic to the complement of the link $\bar\tau\cup\bar{C}$ on $T^1(\Sigma_{0, 2n+3}).$ Moreover, $\bar\tau$ and $\bar{C}$ projects injectively to a closed curve $\tau$ and $C$ on $\Sigma_{0, 2n+3}$ (see Figure \ref{pun2}). Indeed, if we take the diagram on the left of Figure  \ref{pun1}, by placing the link into two perpendicular projection planes, the vertical containing the vertical annuli and the horizontal $B$ and $\bar\tau,$ then there is a circle fibration of the interior of the disk bounded by $B$ by perpendicular circles, where the boundary of the vertical annuli correspond to fibers.
		
		\begin{figure}[h]
			\begin{center} 
				\includegraphics[scale=0.3] {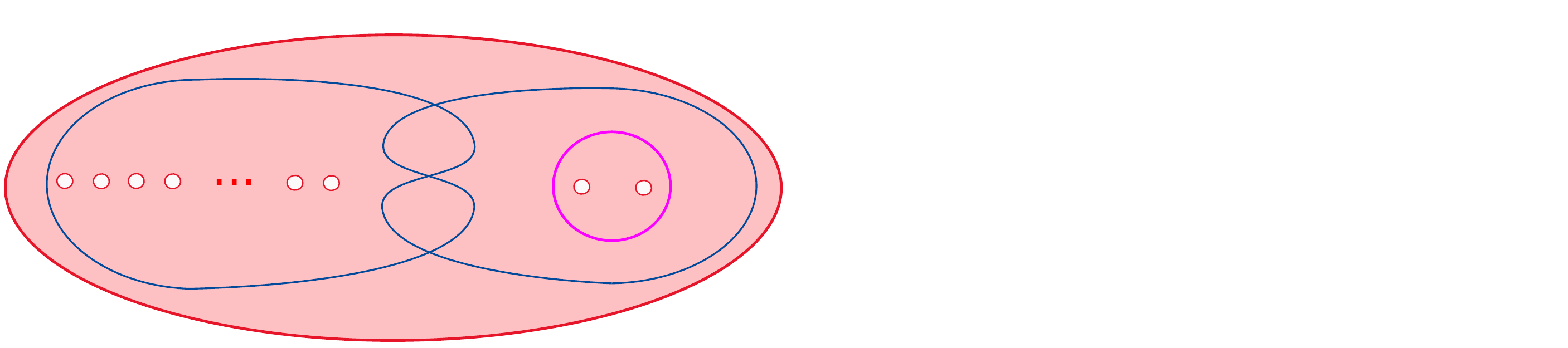}
				\caption{The closed curve $\tau$ and $C$ on $\Sigma_{0, 2n+3}$ (left) and the vector field on $\Sigma_{0, 2n+3}$ that induces a section in $T^1(\Sigma_{0, 2n+3})$ (right).}\label{pun2}
			\end{center} 
		\end{figure}
		\vskip .2cm
		We will construct the last knot component of $L_\gamma,$ denoted by $\overline{\sigma_\gamma}$ inside a normal neighborhood of a section in $T^1(\Sigma_{0, 2n+3})$ such that the vector field on $\Sigma_{0, 2n+3}$ that defines it has\string:
  
  \begin{enumerate}
      \item index $-2n-1$ with respect the tangent vector field on the boundary corresponding to $B;$
      \item  index $0$ with respect the tangent vector field on the boundary relative to the boundary annuli of all $A_{X_i}$ and also $A_{Y_1}.$
  \end{enumerate}

  So it is enough to draw its projection closed curve, denoted by $\sigma_\gamma,$ on $\Sigma_{0, 2n+3}$ and specify the crossing information with itself.
		\vskip .2cm
		We start by marking intervals $I_{X_i}$ and $I_{Y_1}$ on $\Sigma_{0, 2n+3}$ whose preimage under the projection maps are the corresponding vertical annuli $A_{X_i}$ and $A_{Y_1}.$
		\vskip .2cm 
		\begin{figure}[h]
			\begin{center} 
				\includegraphics[scale=1] {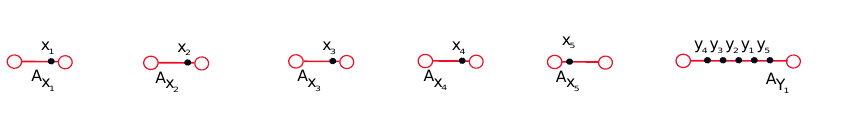}
					\caption{The marking of the intervals associated to  $X^{k_5}YX^{k_4}YX^{k_3}YX^{k_2}YX^{k_1}Y.$}\label{mark}
			\end{center} 
		\end{figure}
		\begin{enumerate}
			
   \item From each  $I_{X_i}$ with $1\leq i\leq n-1$ we draw an arc $\alpha_i$ starting at $x_i\in I_{X_i}$ to  $y_{i}\in I_{Y_1}$ passing one time through each interval $I_{X_j}$ with $i\leq j\leq n$ (see Figure \ref{lc}). Draw all arcs in a way that are disjoint from each other and are parallel between them (see Figure \ref{lc}).
        \item Construct an arc $\alpha_n$ from the point $x_n\in I_{X_n}$ to the point $y_{n}\in I_{Y_{1}}$ disjoint from the arcs $\alpha_i$ with $1\leq i\leq n-1$ and  parallel to them.
			
			\begin{figure}[h]
				\begin{center} 
					\includegraphics[scale=1] {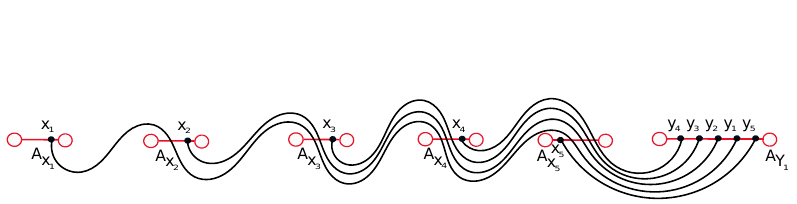}
					\caption{$\alpha_i$ arcs with respect to  $X^{k_5}YX^{k_4}YX^{k_3}YX^{k_2}YX^{k_1}Y.$}\label{lc}
				\end{center} 
			\end{figure}
			\item Connect the point $y_i$ with $x_{i+1}$ with $1\leq i\leq n-1$ with an arc $\beta_i$ disjoint from the intervals and parallel between them (see Figure \ref{lc1}), and connect the point $y_n$ to $x_1$ with an arc $\beta_n$. Construct $\beta_{n-1}$ 
 such that it is the only $\beta$-arc intersecting the $\alpha$-arcs and it precisely only intersects once each arc $\alpha_i$ with $1\leq i\leq n-1.$
		\end{enumerate}

		\begin{figure}[h]
			\begin{center} 
				\includegraphics[scale=1] {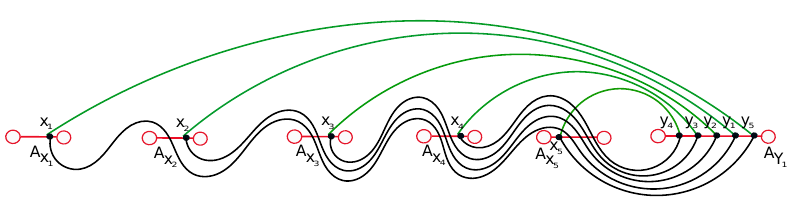}
				\caption{The closed curve $\sigma_\gamma$ associated to  $X^{k_5}YX^{k_4}YX^{k_3}YX^{k_2}YX^{k_1}Y.$}\label{lc1}
			\end{center} 
		\end{figure}\vskip .2cm
		
		Once we join all arcs we will have the closed curve $\sigma_\gamma.$  The crossing information of the link $\overline{\sigma_\gamma}$  is given by the fact that the strand $\beta_{n-1}$ is under the other $\alpha_i$ strands for every $1\leq i\leq n-1$ (see Figure \ref{sigl}).
		
		\begin{figure}[h]
			\begin{center} 
				\includegraphics[scale=1] {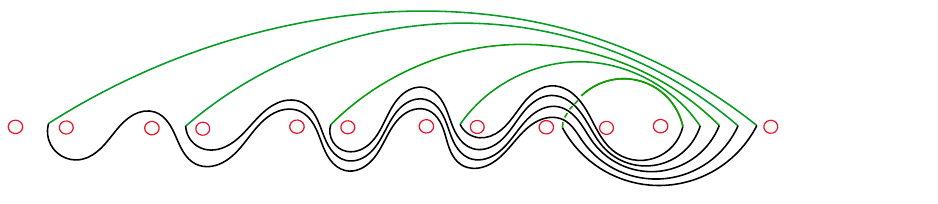}
				\caption{The knot $\overline{\sigma_\gamma}$ associated to $ X^{k_5}YX^{k_4}YX^{k_3}YX^{k_2}YX^{k_1}Y.$}\label{sigl}
			\end{center} 
		\end{figure}
		
		\begin{remark}\label{intnumber}
			The self-intersection number of the closed curve $\sigma_\gamma$ is $n-1.$ 
		\end{remark}

		\begin{claim}\label{augmented link}
			$ M_{\widehat{\gamma}} $ is obtained by applying  annular Dehn filling along the boundary of the vertical annuli components of $ L_\gamma,$ trivial Dehn filling on $B$ and $1$-Dehn filling on $\Bar{C}$
		\end{claim}

		\begin{proof}[\bf{Proof of claim:}]
			First we isotope $\overline{\sigma_\gamma}$ such that the projection of it to a plane parallel to the vertical annuli (see Figure \ref{tt2}) is in a suitable position for the later description of the Dehn surgeries.

			\begin{enumerate}
				\item For the vertical annuli $ A_{X_i} $ with $ i \leq n-1 $ we will do an annular Dehn filling of type $ \frac{1}{k_{n+1-i}-k_{n-i}} $ (see Figure \ref{axi}).

    \begin{figure}[h]
			\begin{center} 
				\includegraphics[scale=.5] {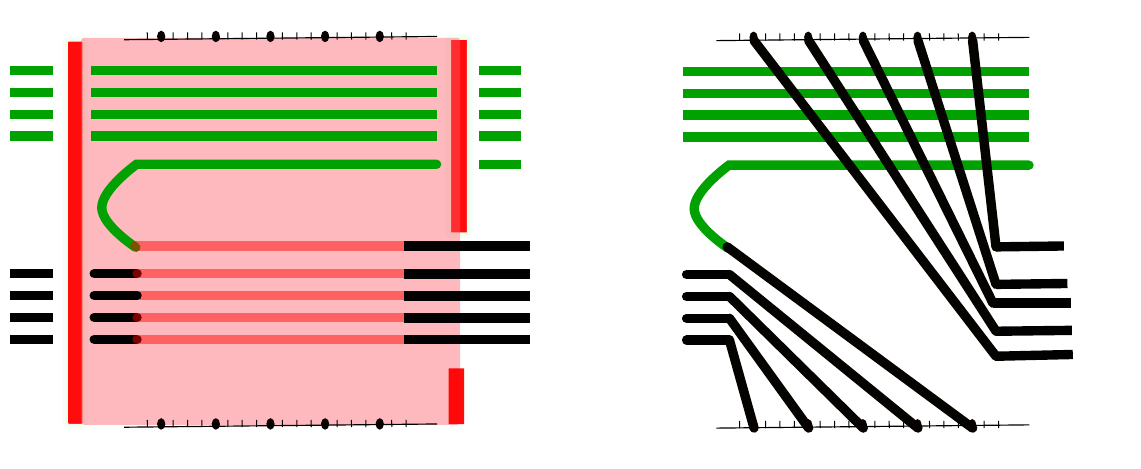}
				\caption{Before (left) and after (right) annular Dehn filling on the vertical annulus $ A_{X_i} $ for $1\leq i\leq n-1$ when $k_{n+1-i}-k_{n-i}=1.$ }\label{axi}
			\end{center} 
		\end{figure}

		\item For $ A_{X_{n}} $ we will do an annular Dehn filling of type $ \frac{1}{k_1}$ (see Figure \ref{axn}).

     \begin{figure}[h]
			\begin{center} 
				\includegraphics[scale=.5] {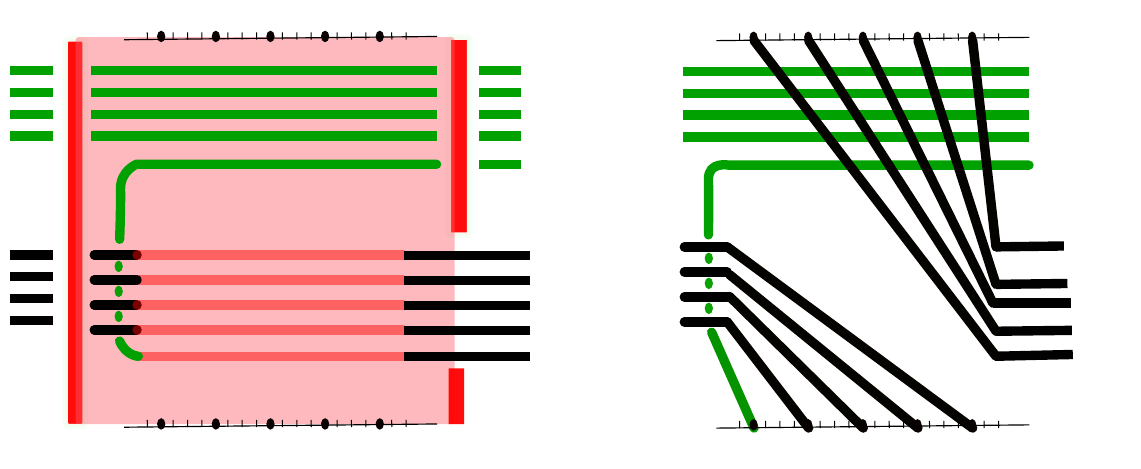}
				\caption{Before (left) and after (right) annular Dehn filling on the vertical annulus $ A_{X_n}$ when $k_1=1$}\label{axn}
			\end{center} 
		\end{figure}

				\item For $A_{Y_1}$ we will do an annular Dehn filling of type $1$ (see Figure \ref{ay}).

    \item For the crossing circle $\Bar{C}$ we do a Dehn filling of type $1 $ (see Figure \ref{ay}).
    
				 \begin{figure}[h]
			\begin{center} 
				\includegraphics[scale=.5] {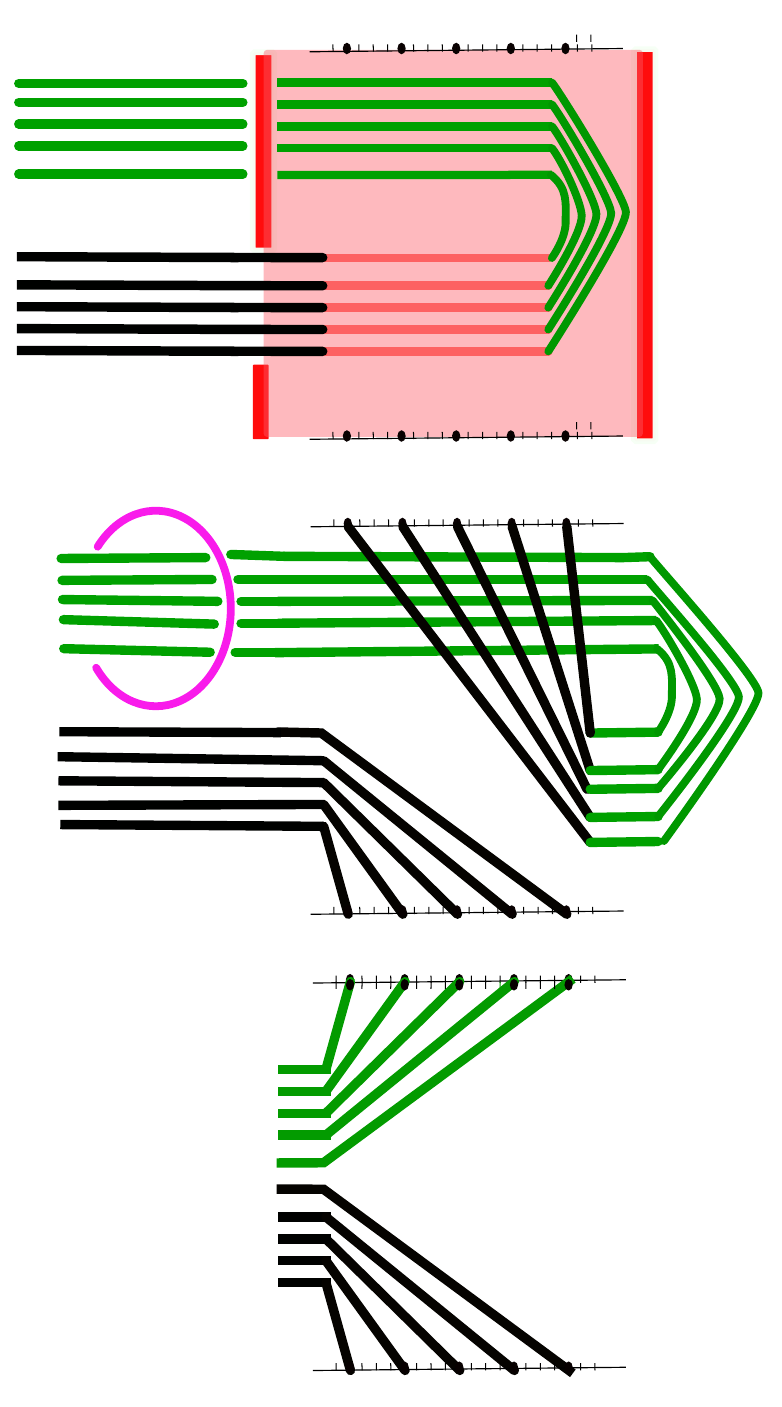}
				\caption{Before (top) and after (central) annular Dehn filling on the vertical annuli  $A_{Y_1}.$ Finally, (bottom) after Dehn filling on the crossing circle $\Bar{C}.$}\label{ay}
			\end{center} 
		\end{figure}
   \item For the knot component $B$ we do a trivial Dehn filling.
			\end{enumerate}

			Notice that each $A_{X_i}$ is intersected by the arc $\beta_i .$ If we make annular Dehn fillings in the order given by $i,$ then the effect of the previous Dehn fillings produce parallel overcrossing strands whose slopes are the following\string:

			\begin{enumerate}
				
				\item For the vertical annuli $ A_{X_i} $ with $ i \leq n-1 $ we will obtain exactly $i-1$ overcrossing strands with slope $\frac{-1}{i-1},$ $i(k_{n+1-i}-k_{n-i})-i$ overcrossing strands with the slope $\frac{-1}{i}$ and $i$ strands going to the next annulus $ A_{X_i+1}.$
				\item For $ A_{X_{n}}, $ obsreve that the undercrossing strand $\bar{\beta}_{n-1}$ intersecting the annulus $ A_{X_{n}}$ passes under all the overcrossing strands $\bar{\alpha}_i$ before the annular Dehn filling (see Figure \ref{axn}-left). Then we will obtain exactly $n-1$ overcrossing strands coming from $ A_{X_{n-1}} $ plus $nk_1-n$ overcrossing strands obtained after annular Dehn filling on $ A_{X_{n}}. $ So we conclude that in total we get $nk_1-1$ overcrossing strands with slope $\frac{-1}{n}$ and $n$ strands going to the next annulus   $A_{Y_1}.$

    \item Finally the effect of an annular Dehn filling of type $1$ in $A_{Y_1}$ is to get $n$ overcrossing strands with slope $\frac{-1}{n}$ and $n$ undercrossing strands with a  $n$-strand full-twist which is removed by doing Dehn filling of type $1$ on the crossing circle $\Bar{C}.$
			\end{enumerate}
			
\begin{figure}[h]
	\begin{center} 
		\includegraphics[scale=0.7] {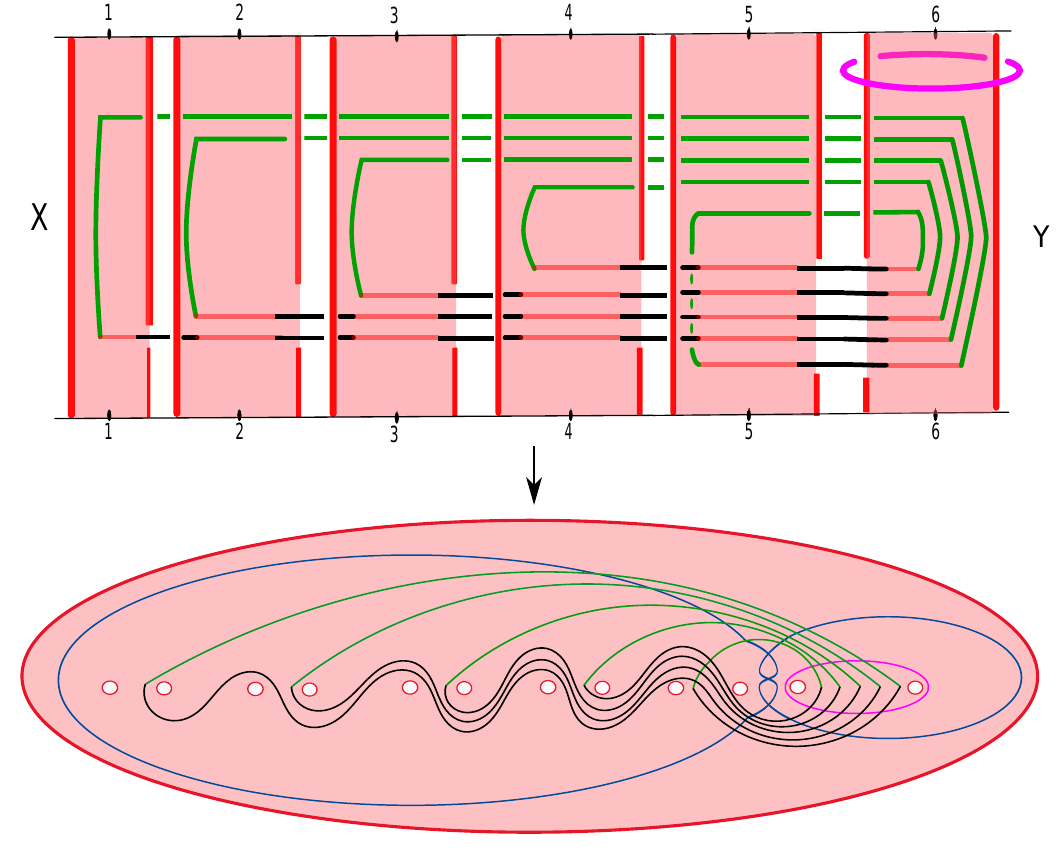}
		\caption{Before (top) annular Dehn filling along all the vertical annuli and crossing circle $\Bar{C}$ on the  Lorenz braid $ \langle 1^{s_1},2^{s_2},...,(n-1)^{s_{n-1}},n^{s_n}\rangle_X $ in Lemma \ref{para}. And the projection of the link $\bar\tau\cup\bar{C}\cup\overline{\sigma_\gamma}$ on $\Sigma_{0, 2n+3}$ (bottom). }\label{tt2}
	\end{center} 
\end{figure}
		\end{proof}

		Finally, let $\Gamma=\tau\cup \sigma_\gamma\cup C,$ by \cite{Thu79} $$\Vol(M_{\widehat\gamma})\leq v_3\|T^1(\Sigma_{0, 2n+3})_{\overline{\Gamma}}\|,$$ and by (\cite{CR18}, Thm 1.5) $$\|T^1(\Sigma_{0, 2n+3})_{\overline{\Gamma}}\|\leq 8i(\Gamma,\Gamma)\leq 8(7n+2),$$
		giving us an upper bound of the volume depending linearly only on $n.$

	\end{proof}
	
Using again \cite{Thu79} to trivially Dehn fill $\bar{\tau}$ in Theorem \ref{seq}, we have the following upper bound for some sequences of Lorenz knots $K_n$ in $\mathbb{S}^3.$ 

\begin{repcorollary}{seqL}
	
	 	There exist a sequence $\{K_n\}$ of Lorenz knots  in $\mathbb{S}^3$ such that $n$ is the braid index of $K_n,$ and
 	$$  \Vol(\mathbb{S}^3\setminus K_n)\leq 8v_3(7n+2),$$
 	where $v_3$ is the volume of a regular ideal tetrahedron. If $K_n$ is not hyperbolic, $  \Vol(\mathbb{S}^3\setminus K_n)$ is the sum of the volumes of the hyperbolic pieces of $\mathbb{S}^3\setminus K_n.$
	\qed
\end{repcorollary}
\begin{note}
	Notice that Corollary \ref{seq} shows that the upper bound for Lorenz knots, with braid index $n$ greater than $33,$ found in (\cite{CFKNP11} , Theorem 1.7) is not sharp. Nevertheless, in general the braid index of a knot or link gives no indication of its volume's complement. For example, in \cite{CFKNP11} they show there exist Lorenz knots with arbitrarily large braid index and yet bounded volume. The reverse result is also known, for example, closed $3$-braids can have unbounded volume \cite{FKP10}.
\end{note}
Finally, as a consequence of Theorem 1.4 in \cite{Rod20}, we have that up to a constant, Theorem \ref{seq} is sharp for a sequence of closed geodesics on the modular surface. 

\begin{reptheorem}{ub}
For the modular surface $\Sigma_{mod},$ there exist a sequence $\{\gamma_n\}$ of closed geodesics on $\Sigma_{mod}$ such that $n$ is half the period of the continued fraction expansion of $\gamma_n,$ and
$$  v_3 \frac{n}{12} \leq \Vol(M_{\widehat{\gamma_n}})\leq 8v_3(7n+2),$$
where $v_3$ is the volume of a regular ideal tetrahedron. 
\end{reptheorem}
\begin{proof}[\bf{Proof:}]
The sequence can be obtained for any infinite subsequence of the following sequence of closed geodesics\string:
$$\prod\limits_{i=1}^{n}(X^{6k_i+1}Y)\hspace{.5cm}\mbox{where}  \hspace{.5cm}\{k_i\}\in \mathbb{N}^\mathbb{N}, \hspace{.5cm}\mbox{and}  \hspace{.5cm}k_i<k_{i+1}.$$

The upper bound for the volume of the corresponding canonical lift complements is a consequence of Theorem \ref{seq} and the lower bound is proven in  (\cite{Rod20}, Theorem 1.4).

\end{proof}

\section{Sequences of closed geodesics on punctured hyperbolic surfaces whose canonical lift complement volumes is bounded by the length}

In this Section we lift the sequences of closed geodesics on the modular surface found in Theorem \ref{seq} and Corollary \ref{ub} to any punctured hyperbolic surface and compare the volume of the corresponding canonical lift complement with the geodesic length of the lift.

\begin{lemma}\label{cover}
     Given any punctured hyperbolic surface of genus $g$ with $k$ punctures $\Sigma,$ there exist a finite covering map from $\Sigma$ to the modular surface of degree  $6(2g+k-2)$. 
\end{lemma}

The Lemma \ref{cover} is probably folklore
knowledge to experts. However, it might not be easy to extract from the literature. For this reason, and for the sake of being more self-contained, we give the proof as follows\string:

\begin{proof}[\bf{Proof of Lemma \ref{cover}:}]
We split the proof into two cases, for surfaces of genus zero and for surfaces of genus grater than zero.

\begin{enumerate}
    \item Given any punctured hyperbolic sphere with $k$ punctures $\Sigma,$ there exist a finite covering map from $\Sigma$ to the modular surface of degree  $6(k-2)$.
    \begin{enumerate}
        \item If $k=3$ consider the subroup of $\PSL_2(\mathbb{Z})$ generated by\string:
$$\footnotesize{\left(\begin{array}{cc}
1&2\\
0&1
\end{array}\right) 
\hspace{.1cm}\mbox{and} \hspace{.1cm} 
\left(\begin{array}{cc}
1&0\\
2&1
\end{array}\right).}$$ 

It is a subgroup of index $6$ over $\PSL_2(\mathbb{Z})$ and the quotient of $\mathbb{H}^2$ under the action of this subgroup by hyperbolic isometries is homeomorphic to a thrice-punctured sphere.
\item If $k>4$ a rotation through the two cusps in a symmetrically arranged thrice-punctured sphere (see Figure \ref{st}-left) yields a $(k-2)$-cover of the thrice-punctured sphere. 
    \end{enumerate}

\begin{figure}[h]
				\begin{center} 
					\includegraphics[scale=.3] {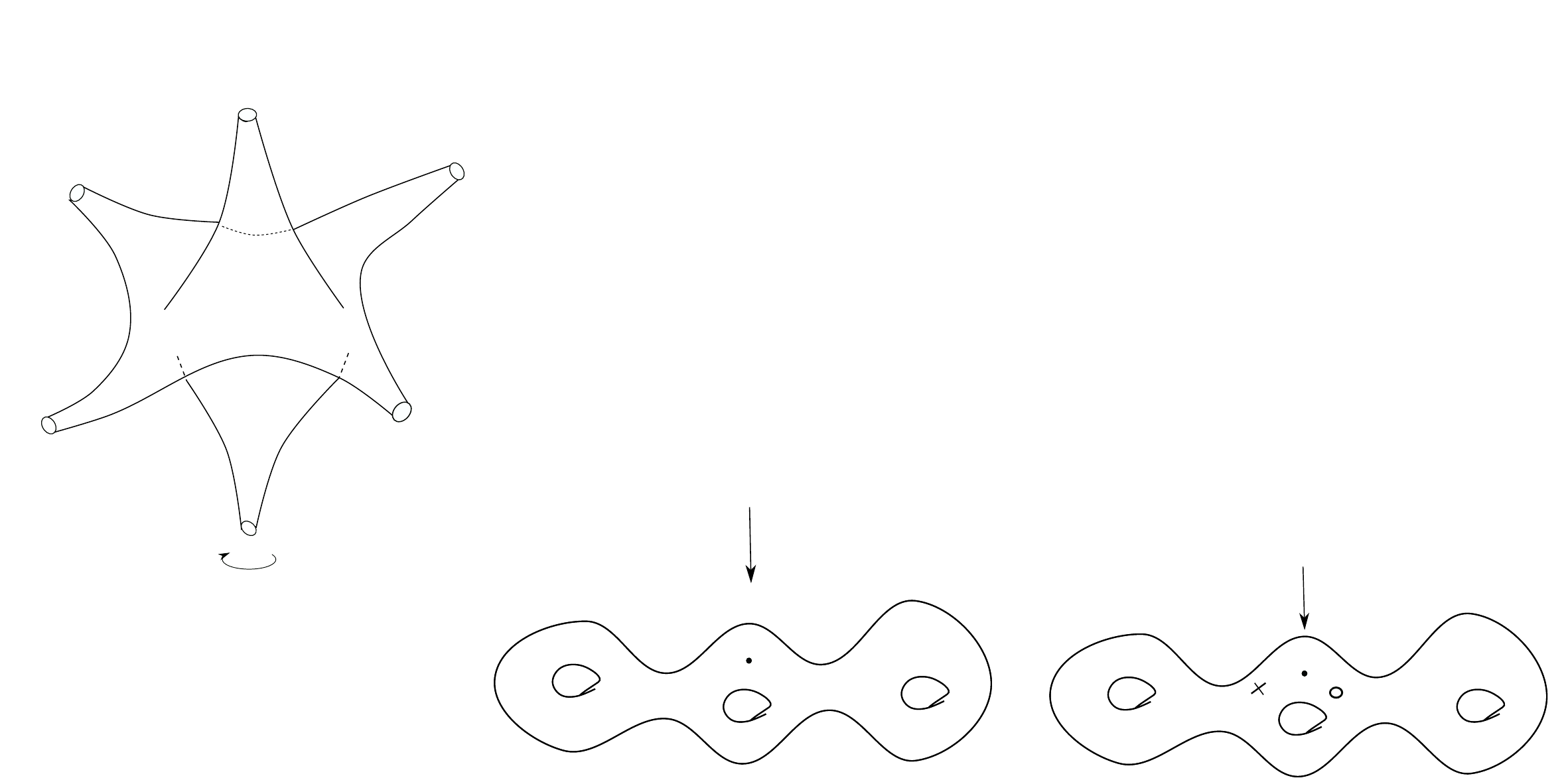}
					\caption{A $4$-cover of a thrice-punctured sphere (left). A branched $5$-cover over the torus, which corresponds to a genus $3$ surface with one conical singularity (center). A branched $7$-cover over the torus, which corresponds to a genus $3$ surface with $3$ conical singularities (right).}\label{st}
				\end{center} 
			\end{figure}
    
\item Given any punctured hyperbolic surface of genus $g\geq1$ with $k$ punctures $\Sigma,$ there exist a finite covering map from $\Sigma$ to the modular surface of degree  $6(2g+k-2)$.
\begin{enumerate}
\item If $g=k=1$ consider the subroup of $\PSL_2(\mathbb{Z})$ generated by\string:
$$\footnotesize{\left(\begin{array}{cc}
2&1\\
1&1
\end{array}\right) 
\hspace{.1cm}\mbox{and} \hspace{.1cm} 
\left(\begin{array}{cc}
2&-1\\
-1&1
\end{array}\right).}$$ It is a subgroup of index $6$ over $\PSL_2(\mathbb{Z})$ and the quotient of $\mathbb{H}^2$ under the action of this subgroup by hyperbolic isometries is homeomorphic to a once-punctured torus.

    \item If $g\geq k=1$ we can construct a $(2g-1)$-cover to the once-punctured torus by considering the square tiled surface in $\mathbb{R}^2$ define by the polygon obtained as the union of unit squares whose centers are in\string:
    $$\{(1/2,-1/2)\}\bigcup_{k=2}^g \biggl\{ \frac{1}{2}\left(4\Bigl \lfloor\frac{k-1}{2}\Bigr\rfloor+1 ,-4\Bigl\lfloor\frac{k}{2}\Bigr\rfloor+1\right), \frac{1}{2}(2k-1,-2k+1)\biggl\}.$$
    Then glue opposite faces by an Euclidean translation (see Figure \ref{st}-center). As each gluing of pair of opposite edges is obtained by a different translation, then the square tiled surface has a unique conical singularity with angle $ 2\pi(2g-1),$ so by the by the Euler-Poincar\'e formula, the topology of the surface after gluing is a genus $g$ surface.
    \vskip .2cm
    This induces a branched $(2g-1)$-cover to the torus with a single branching point, after removing the conical singularities and branched points, we get the covering we wanted.
 \item If $k>1$ and $g\geq 1$ we consider the previous polygone to get a genus $g$ square tiled surface, and add an horizontal strip of $k-1$ unit squares in the right-most vertical edge, if $g$ is even, or a vertical strip of $k-1$ unit squares at the lower-most horizontal edge if $g$ is odd. Then glue opposite faces by an Euclidean translation (see Figure \ref{st}-right). Notice that the gluing  combinatorial pattern has not change, because all the edges in the added strip uses the same translation. After the gluing we would get $k-1$ extra conical singularities each of angle $2\pi$.
\vskip .2cm
 This induces a branched $(2g-1+k)$-cover to the torus with a single branching point, after removing the conical singularities and branched points, we get the covering we wanted.
    
\end{enumerate}
    
\end{enumerate}
\end{proof}

\begin{note}\label{trle}
We briefly recall that in the case $A\in \PSL_2(\mathbb{R}) $ is a hyperbolic element of trace $t,$ the eigenvalues of $A$ are $\frac{-t\pm\sqrt{t^2-4}}{2}. $ Let $\lambda_A$ be the eigenvalue satisfying $|\lambda_A|>1.$ Then, the length of the closed geodesic determine by $A$ is $2\ln|\lambda_A|.$ 
\end{note}
\begin{note}\label{lamb}
In the following calculations we will use the Lambert $W$ function, denoted by  $W:[-e^{-1},\infty)\rightarrow [-1,\infty)$ which is obtained as the inverse function of $f:[-1,\infty)\rightarrow [-e^{-1},\infty)$ defined by $f(x) =xe^x.$ Notice that $W$ is an increasing fucntion and in (\cite{HH08}, Theorem 2.1)  is proven the following useful inequality for every $x \geq e$\string:
$$\frac{\ln(x)}{2}=\ln(\sqrt{x})\leq \ln\left( \frac{x}{\ln(x)}\right)= \ln (x)-\ln( \ln( x)) \leq W(x)\leq \ln(x)-\frac{1}{2}\ln(\ln( x)) \leq \ln(x).$$
\end{note}
\subsection{Volume's upper bound}
Here we prove an upper bound for the volumes of canonical lift complements relative to filling sets of closed geodesics on infinitely many punctured hyperbolic surfaces.

\begin{repcorollary}{nub}
Let $\Sigma$ be a punctured surface of genus $g$ with $k$ punctures, admitting a hyperbolic metric $\rho$ and let $d_\Sigma:=6(2g+k-2)$.   Then there exist a constant $C_\rho>0$ and a sequence $\{{\gamma_n}\}$  of filling finite sets of closed geodesics on $\Sigma$ with at most $d_\Sigma$ elements in each set $\gamma_n$ and $\ell_{\rho}({\gamma_n})\nearrow \infty,$  such that  
$$\Vol(M_{\widehat{{\gamma_n}}})\leq 8d_{\Sigma}v_3\left(\frac{C_\rho\ell_\rho({\gamma_n})}{\ln\left(\frac{\ell_\rho({\gamma_n})}{C_\rho}\right)} +2\right),$$
where $v_3$ is the volume of a regular ideal tetrahedron. 
 \end{repcorollary} 

\begin{proof}[\bf{Proof:}]
First we prove the result for the modular surface case. Let $\gamma_n$ be the unique closed geodesic on $\Sigma_{mod},$ whose corresponding matrix representant is\string:
 
$$A_n:=\prod\limits_{\substack{i=1}}^{n}(X^{k_i}Y),$$

where $(k_i)_{i=1}^{n}\in \mathbb{N}^n$ such that $3^9<k_1+1<k_2,$  $k_i<k_{i+1}$ for $2\leq i\leq n-1,$ $ X=\left(\begin{array}{cc}
	1&1\\
	0&1
	\end{array}\right) 
	$, $
	Y=\left(\begin{array}{cc}
	1&0\\
	1&1
	\end{array}\right),
$
and denote the hyperbolic metric on the modular surface by $\rho_0.$

\begin{claim}\label{qi-l0}
	For all $n\in \mathbb{N}$ we have that\string:
	$$ n\leq \frac{2\ell_{\rho_0}(\gamma_n)}{\ln\left(\frac{\ell_{\rho_0}(\gamma_n)}{e}\right)}. $$
	 
\end{claim}

\begin{proof}[\bf{Proof of claim:}]By  (\cite{BPS16}, Lemma 2.1) we have that,
	$$\ln( n!)\leq  \sum\limits_{\substack{i=1}}^{n} \ln( k_i) \leq  \ell_{\rho_0}(\gamma_n) .$$
	By using the inequality (see \cite{Rob55}),
	
	$$\left(\frac{n}{e}\right)^n\leq n! ,$$
 then $$\left(\frac{n}{e}\right)\ln\left(\frac{n}{e}\right)\leq \frac{\ln( n!)}{e}\leq \frac{\ell_{\rho_0}(\gamma_n)}{e},$$	
	by the definition of the increasing Lambert $W$ function (see Remark \ref{lamb})  we have that\string:
$$\ln\left(\frac{n}{e}\right)= W\left(\ln\left(\frac{n}{e}\right)e^{\ln\left(\frac{n}{e}\right)}\right)=W\left(\left(\frac{n}{e}\right)\ln\left(\frac{n}{e}\right)\right)\leq W\left(\frac{\ell_{\rho_0}(\gamma_n)}{e}\right),$$
so

$$\frac{n}{e}\leq e^{W\left(\frac{\ell_{\rho_0}(\gamma_n)}{e}\right)},$$

because  $W(y)e^{W(y)}=y$ we have that\string: 
$$ n\leq e\frac{\frac{\ell_{\rho_0}(\gamma_n)}{e}}{W\left(\frac{\ell_{\rho_0}(\gamma_n)}{e}\right)} \leq \frac{\ell_{\rho_0}(\gamma_n)}{W\left(\frac{\ell_{\rho_0}(\gamma_n)}{e}\right)}.$$

Also we have that $\ell_{\rho_0}(\gamma_n)>3^9>e^2,$ then we can apply the lower bound of the Lambert W function in Remark \ref{lamb}\string:

	$$ \frac{\ell_{\rho_0}(\gamma_n)}{W\left(\frac{\ell_{\rho_0}(\gamma_n)}{e}\right)}\leq \frac{2\ell_{\rho_0}(\gamma_n)}{\ln\left(\frac{\ell_{\rho_0}(\gamma_n)}{e}\right)}.$$
\end{proof}

By Theorem \ref{seq} and Claim \ref{qi-l0} we have that\string:
$$\Vol(M_{\widehat\gamma_n})\leq 8v_3\left(\frac{14\ell_{\rho_0}(\gamma_n)}{\ln\left(\frac{\ell_{\rho_0}(\gamma_n)}{e}\right)}+2 \right).$$ 

By Lemma \ref{cover} there exist a finite covering map $p,$  of degree  $d_\Sigma:=6(2g+k-2),$ from any punctured hyperbolic surface of genus $g$ with $k$ punctures $\Sigma$  to the modular surface.
\vskip .2cm
Let $\widetilde \gamma_n$ be the finite set of closed geodesics on $\Sigma,$ obtained as the preimage under $p$ of the closed geodesic $\gamma_n$ considered in the modular surface case (see Theorem \ref{seq}). Moreover, if we pullback the hyperbolic metric to  $\Sigma$ under $p,$ we have that $\ell_{p^*(\rho_0)}(\widetilde \gamma_n)=d_\Sigma\ell_{\rho_0}(\gamma_n),$ and denote the metric $p^*(\rho_0)$ as $\rho.$ 
\vskip .2cm
Then by  Theorem \ref{seq},
$$\Vol(M_{\widehat{\widetilde {\gamma_n}}})=d_\Sigma\Vol((T^1\Sigma_{mod})\setminus{\widehat{\gamma_n}})\leq d_{\Sigma}8v_3\left(\frac{14\ell_{\rho_0}(\gamma_n)}{\ln\left(\frac{\ell_{\rho_0}(\gamma_n)}{e}\right)}+2 \right).$$
Finally,
$$\Vol(M_{\widehat{\widetilde {\gamma_n}}})\leq 8d_{\Sigma}v_3\left(\frac{14\ell_{\rho}(\widetilde \gamma_n)}{d_{\Sigma}\ln\left(\frac{\ell_{\rho}(\widetilde \gamma_n)}{d_{\Sigma}e}\right)} +2\right). $$

To proof this result for any hyperbolic metric on $\Sigma,$ follows from the fact that any pair of hyperbolic metrics on a hyperbolic surface are bi-Lipschitz  (see for example \cite{BPS16}, Lemma 4.1).

\end{proof}

\begin{note}Notice that for large $n$ the set  $\widetilde {\gamma_n}$ in Corollary \ref{nub} consists of non-simple geodesics by \cite{Gas16}.
\end{note}

\subsection{Volume's lower and upper bound}
The following result witnesses the sharpness of the length bounds found in the previous Subsection, for volumes of canonical lifts complements relative to some sequence of filling sets of closed geodesics on any punctured hyperbolic surface.

\begin{repcorollary}{2}
Let $\Sigma$ be a punctured surface of genus $g$ with $k$ punctures, admitting a hyperbolic metric $\rho$ and let $d_\Sigma:=6(2g+k-2)$.   Then there exist a constant $C_\rho>0$ and a sequence $\{{\gamma_n}\}$  of filling finite sets of closed geodesics on $\Sigma,$ with at most $d_\Sigma$ elements in each set  $\gamma_n$ and $\ell_{\rho}(\gamma_n)\nearrow \infty,$  such that  
	$$\frac{d_{\Sigma}v_3}{12}\left(\frac{\ell_\rho({\gamma_n})}{C_\rho\ln(C_\rho\ell_\rho({\gamma_n}))} -8\right)  \leq \Vol(M_{\widehat{{\gamma_n}}})\leq 8d_{\Sigma}v_3\left(\frac{C_\rho\ell_\rho({\gamma_n})}{\ln\left(\frac{\ell_\rho({\gamma_n})}{C_\rho}\right)} +2\right).$$ 
	where  $v_3$ is the volume of a regular ideal tetrahedron. 
\end{repcorollary}
\begin{proof}[\bf{Proof:}]
	First we prove the result for the modular surface case. Let $\gamma_n$ be the unique closed geodesic on $\Sigma_{mod},$ whose corresponding matrix representant is\string:
	
	$$A_n:=\prod\limits_{\substack{i=3^8}}^{n+3^8}(X^{6i+1}Y),$$
	
	where $\footnotesize {X=\left(\begin{array}{cc}
		1&1\\
		0&1
		\end{array}\right) 
		\hspace{.2cm}\mbox{and} \hspace{.2cm} 
		Y=\left(\begin{array}{cc}
		1&0\\
		1&1
		\end{array}\right).}
	$
	and denote its hyperbolic metric by $\rho_0.$
	\begin{claim}\label{qi-l}
		For all $n\in \mathbb{N}$ we have that\string:
		$$\frac{\ell_{\rho_0}(\gamma_n)}{14\ln(\ell_{\rho_0}(\gamma_n))}-2\leq n. $$
	\end{claim}
	
	\begin{proof}[\bf{Proof of claim:}]
		Let $\footnotesize {A_n:=\left(\begin{array}{cc}
			a_n&b_n\\
			c_n&d_n
			\end{array}\right)}, $ then\string:
		$$\footnotesize {\left(\begin{array}{cc}
			a_n&b_n\\
			c_n&d_n
			\end{array}\right)= \left(\begin{array}{cc}
			(6n+2)a_{n-1}+(6n+1)c_{n-1}& (6n+2)b_{n-1}+(6n+1)d_{n-1}\\
			a_{n-1}+c_{n-1}&b_{n-1}+d_{n-1}
			\end{array}\right).}$$
		
		Let us denote $z_n:=a_n+b_n+c_n+d_n$ then\string:
		$$z_1=3+2(3^8)<6^6,\hspace{.5cm} z_n\leq 6(n+1)z_{n-1}\hspace{.5cm}\mbox{and} \hspace{.5cm} \operatorname{Trace}{A_n}\leq 6(n+1)z_{n-1}.$$
		Therefore,
		$$ \operatorname{Trace}{A_n} \leq 6^{n+1}(n+1)!z_1\leq 6^{n+1}(n+1)!6^6\leq 6^{n+7}(n+7)!.$$
		Notice that the eigenvalue of $A_n$ whose absolute value is bigger than one, denoted as $\lambda_{A_n},$ is bounded as follows\string:
		$$\frac{\operatorname{Trace}{A_n}}{2}\leq |\lambda_{A_n}|\leq \operatorname{Trace}{A_n}. $$
	Then, by Remark \ref{trle},
		$$  \ell_{\rho_0}(\gamma_n)\leq 2\ln(6^{n+7}(n+7)!)=2((n+7)\ln(6)+\ln((n+7)!) )\leq 2\ln((n+7)!)(1+\ln(6)),$$
  and
  $$ 2\ln((n+7)!)(1+\ln(6))\leq 14\ln((n+7)!).$$
		By using the inequality (see \cite{Rob55}),
		
		$$ n!\leq \left(\frac{n+1}{e}\right)^{n+1},$$
  then,
  $$\frac{\ell_{\rho_0}(\gamma_n) }{14e}\leq \frac{\ln((n+7)!)}{e} \leq  \frac{\ln\left(e^{[(n+8)\ln\left(\frac{n+8}{e}\right)]}\right) }{e} \leq \frac{n+8}{e}\ln\left(\frac{n+8}{e}\right),$$
		
		by the definition of the increasing Lambert $W$ function (see Remark \ref{lamb})  we have that\string:
$$ W\left(\frac{\ell_{\rho_0}(\gamma_n)}{14e}\right)\leq W\left(\left(\frac{n+8}{e}\right)\ln\left(\frac{n+8}{e}\right)\right)= W\left(\ln\left(\frac{n+8}{e}\right)e^{\ln\left(\frac{n+8}{e}\right)}\right) =\ln\left(\frac{n+8}{e}\right),$$
so

$$ e^{W\left(\frac{\ell_{\rho_0}(\gamma_n)}{14e}\right)}\leq\frac{n+8}{e},$$

because  $W(y)e^{W(y)}=y$ we have that\string:

$$\frac{\ell_{\rho_0}(\gamma_n)}{14W\left(\frac{\ell_{\rho_0}(\gamma_n)}{e}\right)}-8\leq \frac{\ell_{\rho_0}(\gamma_n)}{14W\left(\frac{\ell_{\rho_0}(\gamma_n)}{14e}\right)}-8\leq ee^{W\left(\frac{\ell_{\rho_0}(\gamma_n)}{14e}\right)}-8\leq n.$$

Also we have that $\ell_{\rho_0}(\gamma_n)>3^9>e^2,$ then we can apply the upper bound of the Lambert W function in Remark \ref{lamb}\string:

		$$\frac{\ell_{\rho_0}(\gamma_n)}{14\ln\left(\frac{\ell_{\rho_0}(\gamma_n)}{e}\right)}-8\leq\frac{\ell_{\rho_0}(\gamma_n)}{14W\left(\frac{\ell_{\rho_0}(\gamma_n)}{e}\right)}-8.$$

	\end{proof}
	
Let  $\Sigma$ be any  punctured hyperbolic surface of genus $g$ with $k$ punctures.  We use the same sequences of finite sets of closed geodesics $\{\widetilde{\gamma_n}\}$  considered in Corollary \ref{nub}, and the pull-back hyperbolic metric of $\Sigma$ induced by the $d_{\Sigma}$-fold covering map to $\Sigma_{mod}.$
\vskip .2cm
By Theorem \ref{ub}, Corollary \ref{nub} and Claim \ref{qi-l} we have that\string:
$$\frac{d_{\Sigma}v_3}{12}\left(\frac{\ell_{\rho}(\widetilde{\gamma_n})}{14d_{\Sigma}\ln\left(\frac{\ell_{\rho}(\widetilde{\gamma_n})}{ed_{\Sigma}}\right)}-8\right) \leq \Vol(M_{\widehat{\widetilde{\gamma_n}}})\leq 8d_{\Sigma}v_3\left(\frac{14\ell_{\rho}(\widetilde{\gamma_n})}{d_{\Sigma}\ln\left(\frac{\ell_{\rho}(\widetilde{\gamma_n})}{ed_{\Sigma}}\right)}+2\right).$$ 

\end{proof}

\section{On the thrice-punctured sphere case}

In this Section we prove a lower bound for the  volume of canonical lift complements associated to figure-eight type closed geodesics on a thrice-punctured sphere $\Sigma_{0,3}$ (see Subsection \ref{f8g}). Notice that this case is not covered by the lower bound for the volume of canonical lift complements of geodesics on  surfaces in (\cite{Rod20}, Theorem 1.5) because it has a trivial pants decomposition.

\begin{reptheorem}{1}
Given a thrice-punctured sphere $\Sigma_{0,3},$ and $\gamma$ a figure-eight type closed geodesic with respect to $X$ and $Y$ (two free-homotopy classes of distinct punctures in $\Sigma_{0,3}$), we have that\string:
	$$\Vol(M_{\widehat{\gamma}})\geq \frac{v_3}{2}(\sharp\{\mbox{exponents of} \hspace{.2cm}  X\mbox{ in} \hspace{.2cm} \omega_\gamma\}+\sharp\{\mbox{exponents of} \hspace{.2cm}  Y\mbox{ in} \hspace{.2cm} \omega_\gamma\}-2),$$
	where $v_3$ is the volume of the regular ideal tetrahedron, $\omega_\gamma$ is the cyclically reduced word representing the conjugacy class of $\gamma$ in $\langle X,Y\rangle\subset\pi_1(\Sigma_{0,3}).$ 
 \end{reptheorem}
  The lower bound is obtained in terms of combinatorial data coming from the geodesic and an essential simple geodesic arc connecting the remaining puncture to itself $\Sigma_{0,3}.$ 

\vskip .2cm
Given a punctured disk $D$ with a marked point in the boundary $x\in\partial D$ we say that two arcs $\alpha,\beta:[0,1] \rightarrow D$ with $\alpha(\{0,1\})\cup \beta(\{0,1\})\subset \partial D$ are in the same homotopy class in $D,$ if there exist an homotopy $\fun{h}{[0,1]_1\times[0,1]_2}{D}$ such that\string: $$h_0(t_2)=\alpha(t_2),  \hspace{.2cm} h_1(t_2)=\beta(t_2) \hspace{.2cm} \mbox{and} \hspace{.2cm}h([0,1]_1\times\{0,1\})\subset \partial D\setminus \{x\}.$$
Notice that each homotopy class is determine by the winding number relative to the puncture of the closed curve obtained by connecting the ends of the arc along the arc on $\partial D\setminus\{x\}.$

 \begin{note}\label{bda} 
Up to isotopy, for in the family of  arcs without intersection there is only one configuration of such arc in $D$. This is shown in Figure \ref{badarcs} an has winding number $0.$ The $2$ in the lower bound of Theorem \ref{1} comes from the fact that such configurations have at most $1$ homotopy class of $\gamma$-arcs on $D.$
 \begin{figure}[h]
\centering
\includegraphics[scale=0.3] {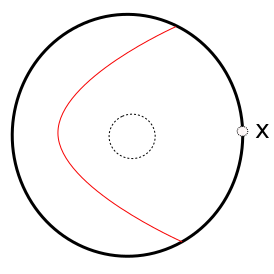}
\caption{ The only  $\gamma$-arc configuration on $D$ up to homotopy classes whose $\gamma$-arcs are simple arcs without intersections.
}\label{badarcs}
\end{figure} 
\end{note}

Before stating the main result to prove Theorem \ref{1}  we recall some definitions.
\vskip .2cm
If $N$ is a hyperbolic $3$-manifold and $S \subset N$ is an embedded incompressible surface, we will use $N\backslash\backslash S$ to denote the manifold that is obtained by cutting along $S;$ it is homeomorphic to the complement in $N$ of an open regular neighborhood of $S.$ If one takes two copies of $N\backslash\backslash S,$ and glues them along their boundary by using the identity diffeomorphism, one obtains the double of $N\backslash\backslash S,$ which is denoted by $d(N\backslash\backslash S).$

\begin{definition}\label{DP}  \normalfont 
Let $D$ be a punctured disk which is induced by spliting $\Sigma_{0,3}$ along a simple geodesic arc $\beta$ connecting a puncture with itself and $\gamma$ a figure-eight type  geodesic on $\Sigma_{0,3}$ such that $D\cap\gamma$ is a finite set of geodesic arcs $\{\alpha_i\}$ connecting $\beta.$ Then denote,
$$D_{\widehat\gamma} \hspace{.2cm} \mbox{as} \hspace{.2cm} T^1D\setminus\bigcup_i \widehat\alpha_i\cong (\mathbb{S}^1\times D)\setminus\bigcup_i \widehat\alpha_i.$$
And define,
$$d({D}_{\widehat\gamma}),$$ as gluing two copies of $T^1D\setminus\bigcup_i \widehat\alpha_i$ along the punctured sphere coming from 

$$ \mathbb{S}^1\times \beta\setminus\big( \bigcup_i  \widehat\alpha_i\big), $$
by using the identity. Moreover, $d({D}_{\widehat\gamma})$ is homeomorphic to
$$(\mathbb{S}^1\times  \Sigma_0)\setminus \bigcup_i  d(\widehat\alpha_i),$$
where $\Sigma_0$ is a thrice-punctured sphere and $d(\widehat\alpha_i)$ is a knot in $\mathbb{S}^1\times  \Sigma_0$ obtained by gluing $\widehat\alpha_i$ along the two points $ (\mathbb{S}^1\times \beta)\cap\widehat\alpha_i$ by the identity.
\end{definition}

\begin{figure}[h]
\centering
\includegraphics[scale=0.4] {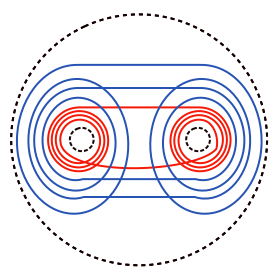}
\caption{The projection of  $d({D}_{\widehat{\gamma}})$ (after an homotopy of $\gamma$-arcs to a minimal position configuration) over $ \Sigma_0.$}\label{DPg}
\end{figure}

Let $M$ be a connected, orientable 3-manifold with boundary and let $S(M;\mathbb{R})$ be the \textit{singular chain complex} of $M.$ More concretely, $S_k(M;\mathbb{R})$ is the set of formal linear combination of $k$-simplices, and we set as usual $S_k(M, \partial M;\mathbb{R})= S_k(M;\mathbb{R})/S_k(\partial M;\mathbb{R})$. We denote by 
$\|c \|$ the $l_1$-norm of  the $k$-chain $c$.
If $\alpha$ is a homology class in $H^{sing}_k(M, \partial M;\mathbb{R})$, the \textit{Gromov norm of $\alpha$}  is defined as\string:
$$\|\alpha\|=\inf_{[c]=\alpha} \{ \|c \|=\sum_{\sigma}|r_\sigma|\hspace{.1cm}\mbox{such that}\hspace{.1cm} c=\sum_{\sigma}r_\sigma \sigma\}.$$
The \textit{simplicial volume} of $M$ is the Gromov norm of the fundamental class of $(M,\partial M)$ in  $H^{sing}_3(M,\partial M;\mathbb{R})$ and is denoted by $\|M\|.$
\vskip .2cm
The key ingredient to prove Theorem \ref{1} is the following result due to Agol, Storm and Thurston (\cite{AST07}, Theorem 9.1)\string:

 \begin{theorem*}[Agol-Storm-Thurston]
 Let $N$ be a compact manifold with interior a hyperbolic $3$-manifold of finite volume. Let $S$ be an embedded incompressible surface in $N.$ Then
 $$\Vol(N)\geq  \frac{v_3}{2} \| d(N\backslash\backslash S)\|.$$
  \end{theorem*}
\begin{definition}\label{split} Let $\beta$ be a simple geodesic arc connecting a puncture with itself  on $\Sigma_{0,3},$ then we define the following embedded surface on $M_{\widehat\gamma}$\string:

$$(T_\beta)_{\widehat\gamma}:=(T_\beta)\setminus \mathcal{N}_{\widehat\gamma} $$

where $T_\beta  $ is the pre-image of $ \beta$ under the map $T^1(\Sigma_{0,3})\rightarrow\Sigma_{0,3}.$
\end{definition} 
 We now prove the lower bound for the volume of the canonical lift complement\string:
 \begin{proof}[\bf{Proof of Theorem \ref{1}}]
Let $\beta$ be a simple geodesic arc connecting a puncture with itself on $\Sigma_{0,3},$ inducing a decomposition on $D_R$ and $D_L.$ Consider the  surface  $ (T_{\beta})_{\widehat\gamma}$ in $M_{\widehat\gamma}$ 

\begin{claim}\label{mer}
The embedded surface $(T_\beta)_{\widehat\gamma}$ in $M_{\widehat\gamma}$ is incompressible.
\end{claim}

 \begin{figure}[h]
\centering
\includegraphics[scale=0.2] {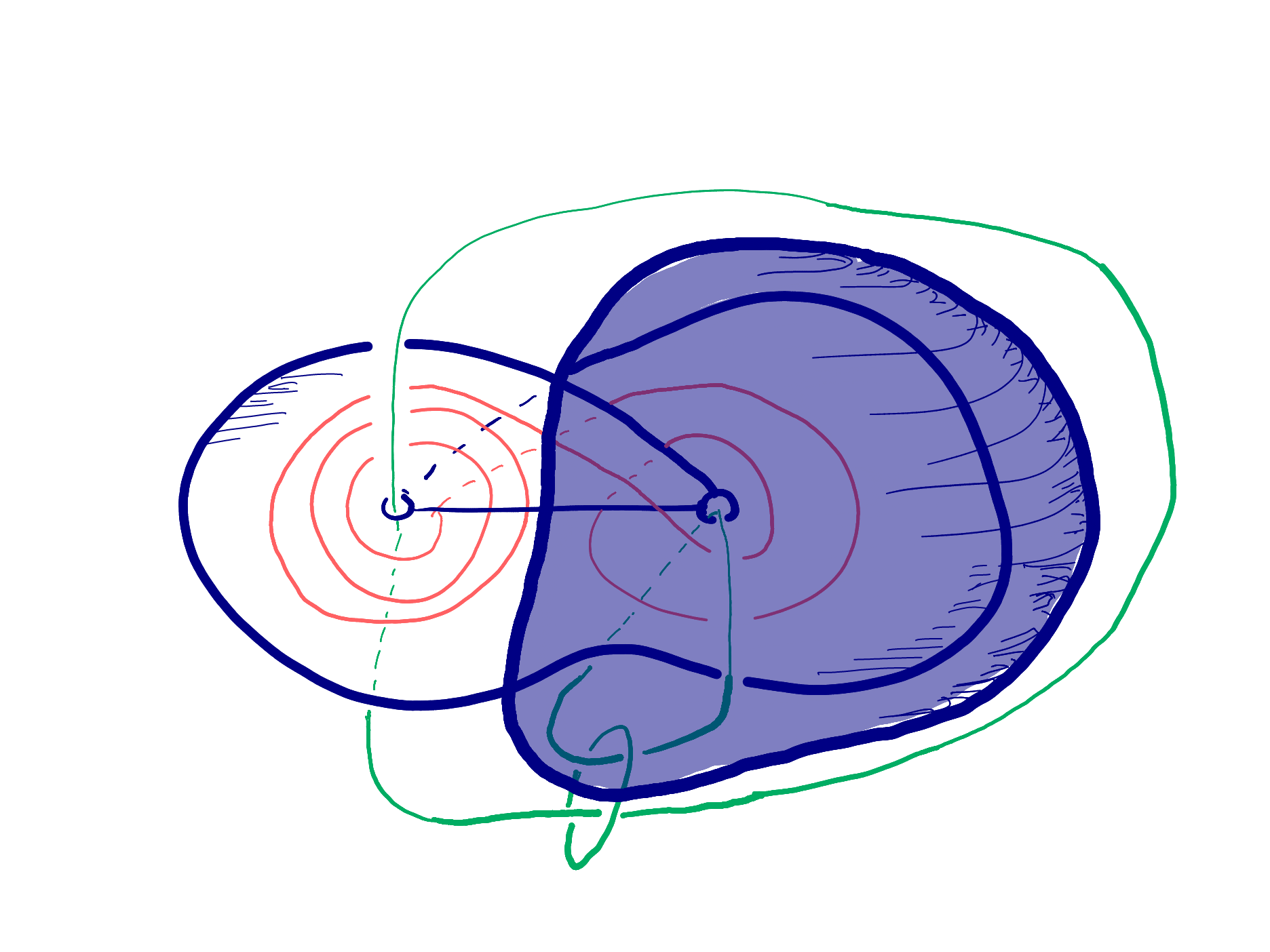}
\caption{ The punctured sphere $(T_\beta)_{\widehat\gamma}.$
}\label{sphe}
\end{figure} 

\begin{proof}[\bf{Proof of Claim \ref{mer}}]
Let us split  $M_{\widehat\gamma}$ by $(T_\beta)_{\widehat\gamma},$ into pieces $D^R_{\widehat\gamma} $ and $D^L_{\widehat\gamma}.$ 
\vskip .2cm
If the surface  $(T_\beta)_{\widehat\gamma}$ is compressible in $D^R_{\widehat\gamma},$ then there is a disk $D_0$ whose interior is in the interior of $D^R_{\widehat\gamma}$ and whose boundary is an essential simple closed curve $\alpha$ in $(T_\beta)_{\widehat\gamma}.$ 
\vskip .2cm
As $T_\beta$ is an incompressible surface in $T^1(\Sigma_{0,3}),$  then  $\alpha$ bounds a disk $D_1$ in $T_\beta$ that intersects ${\widehat\gamma}.$  So by the irreducibility of $T^1(\Sigma_{0,3}),$ the embedded sphere formed by $D_0\cup D_1$  would bound a ball $B$ in $T^1(\Sigma_{0,3}).$
\vskip .2cm
As $\widehat\gamma$ is a knot there is at least one trivial $1$-tangle of $\widehat\gamma$ inside $B$ with endpoints at $D_1.$  The tangles can always be pushed away from $D_0,$ because  they are unknotted inside $B$ as $\gamma$ is in minimal position in $\Sigma_{0,3}.$ Then $\alpha$ would bound a disc in $(T_\beta)_{\widehat\gamma},$ contradicting the fact that $\alpha$ is essential in $(T_\beta)_{\widehat\gamma}.$ 
\vskip .2cm
 A similar argument implies that $(T_\beta)_{\widehat\gamma}$ is compressible in $D^L_{\widehat\gamma}.$
 \end{proof} 

 From (\cite{AST07}, Theorem 9.1) we deduce that\string:
$$\Vol(M_{\widehat\gamma}) \geq\frac{v_3}{2} \| d(M_{\widehat\gamma}\backslash\backslash T_\beta)\|= \frac{v_3}{2}\sum_{P \in \Pi} \| d({D}_{\widehat\gamma})\|.$$
For any piece $D\in\{D_R,D_L\}$ we have\string:
$$v_3\sharp\{\mbox{cusps of}  \hspace{.2cm} d({D}_{\widehat\gamma})^{hyp}\} \leq  \Vol( d({D}_{\widehat\gamma})^{hyp})\leq v_3\| d({D}_{\widehat\gamma})^{hyp}\| = v_3\| d({D}_{\overline{\gamma}})\|$$
 where $d({D}_{\widehat\gamma})^{hyp}$ is the atoroidal piece of $d({D}_{\widehat\gamma}),$   i.e., the complement of the characteristic sub-manifold, with respect to its JSJ-decomposition. The first and second inequality come from \cite{Ada88}  and \cite{Gro82} respectively. 
 \vskip .2cm
 Notice that if $\omega_1$ and $\omega_2$ are a pair of homotopic $\gamma$-arcs on $D$ then their respective canonical lifts  $\widehat{\omega_1}$ and $\widehat{\omega_2}$ are isotopic $\widehat\gamma$-arcs in $T^1D.$ Indeed, let $\widetilde{\omega_1}$ and  $\widetilde{\omega_2}$ be lifts on the universal cover starting in the same fundamental domain. Take an homotopy of geodesics that varies from  $\widetilde{\omega_1}$ to $\widetilde{\omega_2}$ and project this homotopy to $\Sigma_{0,3}.$ This will give us a homotopy $h$ of geodesics arcs $h_t$ that start in $\omega_1$ and end in $\omega_2.$ The image of the geodesic homotopy does not intersects other $\widehat{\gamma}$-arcs because of local uniqueness of the geodesics. Then we have that the geodesic homotopy induces an isotopy in $T^1D$ between their corresponding canonical lifts. Moreover, the image of the geodesic homotopy induces a co-bounding annulus between the corresponding knots coming from the double of the corresponding canonical lifts in $d(T^1D).$  Therefore contributing to a Seifert-fibered component, where the JSJ-decomposition separates this set of parallel knots from the rest of the manifold (see Figure \ref{DPgJSJ}).
 \vskip .2cm
 Let $\Omega$ be the subset of $\gamma$-arcs on $D$ having one arc for each homotopy class of $\gamma$-arcs on $D$. This means that $d({D}_{\widehat\gamma})^{hyp}\cong d({D}_{\widehat{\Omega}})^{hyp}.$ Moreover, $d({D}_{\widehat{ \Omega}})$ can be seen as a link complement in $\mathbb{S}^1\times  \Sigma_0$, see Definition \ref{DP}, whose projection to $\Sigma_0$ is a union of closed loops transversally homotopic to a union closed loops in minimal position. By  using (\cite{CR18}, Theorem 1.3), we have that the atoroidal piece of $d({D}_{\widehat{\Omega}})$ corresponds to  $\Sigma_0$ if $d(\Omega)$ are non-simple closed curves.

  \begin{enumerate}
   \item If the $\Omega$-arc configuration on $D$ is the one of Remark \ref{bda}, then we have that $d({D}_{\widehat\Gamma})^{hyp}=\emptyset$  and Remark \ref{bda} also gives us\string:
 $$v_3(\sharp\{\mbox{homotopy classes of} \hspace{.2cm}  \gamma\mbox{-arcs in} \hspace{.2cm} D\}-1)\leq v_3\sharp\{\mbox{cusps of}  \hspace{.2cm} d({D}_{\widehat\gamma})^{hyp}\}.$$
 \item If the $\Omega$-arc configuration on $D$ is not the one of Remark \ref{bda}, then there is at least one geometric intersection point on the projection of the link complement $d({D}_{\widehat{ \Omega}})$  to $\Sigma_0.$  \end{enumerate} 
\vskip .2cm

\begin{figure}[h]
\centering
\includegraphics[scale=0.4] {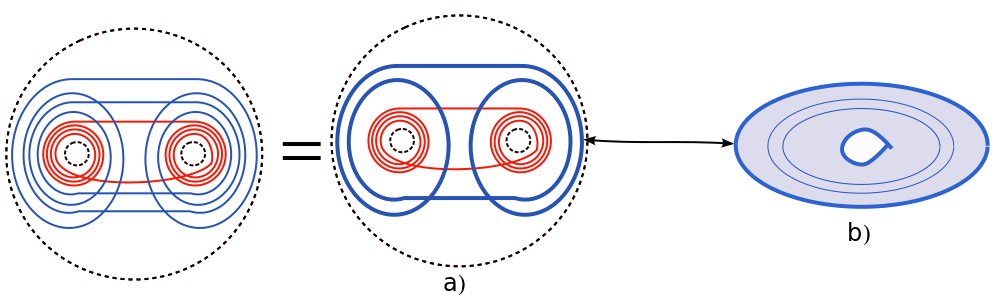}
\caption{The JSJ-decomposition of $d({D}_{\widehat{\gamma}})$ of Figure \ref{DPg}.
}\label{DPgJSJ}
\end{figure}
By (\cite{CR18}, Theorem 1.3) we conclude that $d({D}_{\widehat\gamma})^{hyp}\neq\emptyset.$ We will now define an injective function\string: 
$$\left\{ \begin{array}{l} \widehat\gamma\mbox{-arcs in} \\ \hspace{.5cm} T^1D\end{array}\right\} \overset{\varphi}\longrightarrow \left\{ \begin{array}{l}\hspace{.1cm} \mbox{cusps of} \\ d({D}_{\widehat\gamma})^{hyp}\end{array}\right\}$$
where the target can be decomposed as\string:
$$\left\{ \begin{array}{l}\hspace{.1cm} \mbox{cusps of} \\ d({D}_{\widehat\gamma})^{hyp}\end{array}\right\}=\left\{ \begin{array}{l}\hspace{.3cm}\mbox{splitting tori of the} \\ \mbox{JSJ-decomposition of}  \\ \hspace{1.3cm} d({D}_{\widehat\gamma})\end{array}\right\}\amalg\left\{ \begin{array}{l}\hspace{.8cm}\mbox{cusp in}  \\   d({D}_{\widehat\gamma}) \cap  d({D}_{\widehat\gamma})^{hyp}\end{array}\right\}$$ 
The function $\varphi$ is defined as follows: if the cusps in $d({D}_{\widehat\gamma})$ are induced by the $\widehat\gamma$-arc in $T^{1}D$ belonging to the characteristic sub-manifold of $d({D}_{\widehat\gamma}),$  $\varphi$ maps it to a splitting tori connecting the hyperbolic piece with the component of the characteristic sub-manifold where it is contained. Otherwise, the cusp belongs to $d({D}_{\widehat\gamma})^{hyp}$ and $\varphi$ sends it to itself, see Figure \ref{DPgJSJ}. 
Assume that there are more isotopy classes of $\widehat\gamma$-arcs in $T^{1}D$ than the number of cusps of $d({D}_{\widehat\gamma})^{hyp}$. Then, there are two tori, associated with non-isotopic $\widehat\gamma$-arcs in $T^{1}D,$ that belong to the same connected component of the characteristic sub-manifold. Since each component of the characteristic sub-manifold is a Seifert-fibered space over a punctured surface we have that all such arcs correspond to regular fibres. Thus, they are isotopic in the corresponding component hence isotopic in $T^{1}D,$ contradicting the fact that they were not isotopic.   
\vskip .2cm
Finally, since two isotopic $\widehat\gamma$-arcs in $T^1D$ induce a homotopy between their projections in $D.$ Then for the case for $d({D}_{\widehat\gamma})^{hyp}\neq\emptyset,$ we have that\string:

 $$ v_3\sharp\{\mbox{homotopy classes of} \hspace{.2cm}  \gamma\mbox{-arcs in} \hspace{.2cm} D\}\leq v_3\sharp\{\mbox{cusps of}  \hspace{.2cm} d({D}_{\widehat\gamma})^{hyp}\},$$ as homotopy classes of $\gamma$-arcs in $D_R$ (resp. $D_L)$ is given by the winding numbers obtained from the exponents of $X$ (resp. $Y)$ relative to the primitive word in the semigroup generated by $X$ and $Y$ representing $\gamma.$ Then,

  $$ \sharp\{\mbox{homotopy classes of} \hspace{.2cm}  \gamma\mbox{-arcs in} \hspace{.2cm} D_R\}=\sharp\{\mbox{exponents of} \hspace{.2cm}  X\mbox{ in} \hspace{.2cm} \omega_\gamma\},$$
 and
 $$\sharp\{\mbox{homotopy classes of} \hspace{.2cm}  \gamma\mbox{-arcs in} \hspace{.2cm} D_L\}=\sharp\{\mbox{exponents of} \hspace{.2cm}  Y \mbox{ in} \hspace{.2cm} \omega_\gamma\}.$$
 \end{proof}

As a consequence of Theorem \ref{1} we prove a version of Theorem \ref{2} for sequences of closed geodesics, without using the finite covering argument\string:

 \begin{corollary}\label{tps}
 
	Given a hyperbolic metric $\rho$ on $\Sigma_{0,3},$   then there exist  constants $C_\rho,\delta_{\rho}>0$ and a sequence $\{\gamma_n\}$ of closed geodesics  in $\Sigma_{0,3},$ with $\ell_{\rho}(\gamma_n)\nearrow \infty$  such that $$\frac{v_3}{2}\left(\frac{\frac{\ell_{\rho}(\gamma_n)-\delta_{\rho}}{C_{\rho}}}{\ln(C_{\rho}\ell_{\rho}(\gamma_n))}-4\right)\leq\Vol(M_{\widehat\gamma_n})\leq 8v_3\left(\frac{7C_{\rho_0}(\ell_{\rho_0}(\gamma_n)+\delta_{\rho_0})}{\ln\left(\frac{\ell_{\rho_0}(\gamma_n)}{C_{\rho_0}}\right)}+10\right),$$ 
	where $v_3$ is the volume of a regular ideal tetrahedron, and $C_\rho,\delta_{\rho}$ depend on the metric $\rho$ and $\gamma_0.$
 \end{corollary}

\begin{proof}[\bf{Proof:}]
We will start proving the result for a particular hyperbolic metric $\rho_0$ on $\Sigma_{0,3},$ by fixing the following representation $\fun{\psi}{\pi_1(\Sigma_{0,3}):=\langle x,y\rangle}{\PSL_2(\mathbb{R}) }$ such that
$$\footnotesize {\psi(x)=X=\left(\begin{array}{cc}
1&2\\
0&1
\end{array}\right) 
\hspace{.5cm}\mbox{and} \hspace{.5cm} 
\psi(y)=Y=\left(\begin{array}{cc}
1&0\\
2&1
\end{array}\right).}
$$
Notice that $x$ and $y$ represent the free homotopy class of two different punctures of $\Sigma_{0,3}.$ Let $\gamma_k$ be the unique closed geodesic on $\Sigma,$ whose corresponding matrix representant under $\psi$ is\string:
 
$$A_n:=\prod\limits_{i=1}^{n}(X^{mi+r}Y)\hspace{.2cm}\mbox{where} \hspace{.2cm} m>3^9\hspace{.2cm}\mbox{and} \hspace{.2cm} 0\leq r<m.$$

\begin{claim}\label{l3}
For all $n\in \mathbb{N}$ we have that\string:
$$\frac{\frac{\ell_{\rho_0}(\gamma_n)-\delta_{\rho_0}}{C_{\rho_0}}}{\ln(C_{\rho_0}\ell_{\rho_0}(\gamma_n))}-1\leq n\leq \frac{C_{\rho_0}(\ell_{\rho_0}(\gamma_n)+\delta_{\rho_0})}{\ln\left(\frac{\ell_{\rho_0}(\gamma_n)}{C_{\rho_0}}\right)}+1 $$
where $C_{\rho_0}:=em$ and $\delta_{\rho_0}:=2\ln\left(3+2(m+r)\right).$
 \end{claim}

\begin{proof}[\bf{Proof of claim:}]
Let $\footnotesize {A_n:=\left(\begin{array}{cc}
a_n&b_n\\
c_n&d_n
\end{array}\right)}, $ then\string:
$$\footnotesize {\left(\begin{array}{cc}
a_n&b_n\\
c_n&d_n
\end{array}\right)= \left(\begin{array}{cc}
(4(mn+r)+1)a_{n-1}+2(mn+r)c_{n-1}& (4(mn+r)+1)b_{n-1}+2(mn+r)d_{n-1}\\
2a_{n-1}+c_{n-1}&2b_{n-1}+d_{n-1}
\end{array}\right). }$$

Let us denote $z_n:=a_n+b_n+c_n+d_n$ then\string:
$$z_1=3+2(m+r),\hspace{.1cm}(2mn)z_{n-1}\leq z_n\leq 4m(n+1)z_{n-1}\hspace{.1cm}\mbox{and} \hspace{.1cm}z_{n-1}\leq \operatorname{Trace}{A_n}\leq 4m(n+1)z_{n-1}.$$
Therefore,
$$ (2m)^{n-2}(n-1)!z_1\leq \operatorname{Trace}{A_n} \leq (4m)^{n}(n+1)!z_1 .$$
Notice that the eigenvalue of $A_n$ whose absolute value is bigger than one, denoted as $\lambda_{A_n},$ is bounded as follows\string:
$$\frac{\operatorname{Trace}{A_n}}{2}\leq |\lambda_{A_n}|\leq \operatorname{Trace}{A_n}. $$
Finally, by the Remark \ref{trle},
$$2\ln((n-1)!)\leq 2\ln\left(\frac{1}{2}(2m)^{n-2}(n-1)!z_1\right) \leq \ell_{\rho_0}(\gamma_n) $$

$$\ell_{\rho_0}(\gamma_n) \leq 2\ln((4m)^{n}(n+1)!z_1)= 2(n\ln(4m)+\ln((n+1)!)+\ln(z_1))\leq(2\ln(4m)+2)\ln((n+1)!)+2\ln(z_1)),$$
in addition we have that $m>3^9$ so,
$$\ell_{\rho_0}(\gamma_n) \leq (2\ln(4m)+2)\ln((n+1)!)+2\ln(z_1)\leq m \ln((n+1)!)+\delta_{\rho_0}$$

	By using the inequalities (see \cite{Rob55}),
	
	\begin{equation}\label{a}\left(\frac{n}{e}\right)^n\leq n!\leq \left(\frac{n+1}{e}\right)^{n+1},\end{equation}

For the left inequality in \ref{a} we have that\string:
$$\frac{n-1}{e}\ln\left(\frac{n-1}{e}\right)\leq\frac{\ln((n-1)!)}{e}\leq\frac{\ell_{\rho_0}(\gamma_n)}{2e}$$

		by the definition of the increasing Lambert $W$ function (see Remark \ref{lamb})  we have that\string:
$$\ln\left(\frac{n-1}{e}\right)= W\left(\ln\left(\frac{n-1}{e}\right)e^{\ln\left(\frac{n-1}{e}\right)}\right)=W\left(\left(\frac{n-1}{e}\right)\ln\left(\frac{n-1}{e}\right)\right)\leq W\left(\frac{\ell_{\rho_0}(\gamma_n)}{2e}\right),$$
so

$$\frac{n-1}{e}\leq e^{W\left(\frac{\ell_{\rho_0}(\gamma_n)}{e}\right)},$$

because  $W(y)e^{W(y)}=y$ we have that\string:

 $$n\leq ee^{W\left(\frac{\ell_{\rho_0}(\gamma_n)}{e}\right)} +1 \leq\frac{\ell_{\rho_0}(\gamma_n)}{W\left(\frac{\ell_{\rho_0}(\gamma_n)}{e}\right)}+1.$$

Also we have that $\ell_{\rho_0}(\gamma_n)>3^9>e^2,$ then we can apply the lower bound of the Lambert W function in Remark \ref{lamb}\string:

$$\frac{\ell_{\rho_0}(\gamma_n)}{W\left(\frac{\ell_{\rho_0}(\gamma_n)}{e}\right)}+1\leq \frac{2\ell_{\rho_0}(\gamma_n)}{\ln\left(\frac{\ell_{\rho_0}(\gamma_n)}{e}\right)}+1\leq 
\frac{em(\ell_{\rho_0}(\gamma_n)+\delta_{\rho_0})}{\ln\left(\frac{\ell_{\rho_0}(\gamma_n)}{em}\right)}+1. $$

For the right inequality in \ref{a} we have that\string:
$$\ell_{\rho_0}(\gamma_n) \leq m\ln((n+1)!)+\delta_{\rho_0}\leq m\ln\left(e^{[(n+2)\ln\left(\frac{n+2}{e}\right)]}\right)+\delta_{\rho_0}
\leq\left(em\frac{n+2}{e}\right) \ln\left(\frac{n+2}{e}\right)+\delta_{\rho_0}.$$

By the definition of the increasing Lambert $W$ function (see Remark \ref{lamb})  we have that\string:
$$ W\left(\frac{\ell_{\rho_0}(\gamma_n)-\delta_{\rho_0}}{em}\right)\leq W\left(\left(\frac{n+2}{e}\right)\ln\left(\frac{n+2}{e}\right)\right)= W\left(\ln\left(\frac{n+2}{e}\right)e^{\ln\left(\frac{n+2}{e}\right)}\right) =\ln\left(\frac{n+2}{e}\right),$$
so

$$e^{W\left(\frac{\ell_{\rho_0}(\gamma_n)-\delta_{\rho_0}}{em}\right)}\leq\frac{n+2}{e} ,$$

because  $W(y)e^{W(y)}=y$ we have that\string:

$$ \frac{\ell_{\rho_0}(\gamma_n)-\delta_{\rho_0}}{mW(em\ell_{\rho_0}(\gamma_n))}-2\leq \frac{\ell_{\rho_0}(\gamma_n)-\delta_{\rho_0}}{mW\left(\frac{\ell_{\rho_0}(\gamma_n)-\delta_{\rho_0}}{em}\right)}-2\leq ee^{W\left(\frac{\ell_{\rho_0}(\gamma_n)-\delta_{\rho_0}}{em}\right)}-2\leq n.$$

Also we have that $\ell_{\rho_0}(\gamma_n)>1/e,$ then we can apply the upper bound of the Lambert W function in Remark \ref{lamb}\string:

$$\frac{\ell_{\rho_0}(\gamma_n)-\delta_{\rho_0}}{em\ln(em\ell_{\rho_0}(\gamma_n))}-2\leq \frac{\ell_{\rho_0}(\gamma_n)-\delta_{\rho_0}}{mW(em\ell_{\rho_0}(\gamma_n))}-2$$

\end{proof}

For the volume upper bound, notice that by adding a crossing circle on the twisted region of the trefoil knot $\overline{\tau}$ (see Section \ref{modular1}). By \cite{Ada85} (see Figure \ref{ada}) we  have that the volume of the complement of $\widehat\gamma,$ the  canonical lift of a figure-eight closed geodesic in $T^1(\Sigma_{0,3}),$  is the same as the volume of the canonical lift complement canonical lift of the corresponding closed geodesic on $\Sigma_{mod}$ and an extra crossing circle in $T^1\Sigma_{mod}.$   By using the same ideas as in Theorem \ref{seq}, we obtain\string:
$$\Vol(M_{\widehat{\gamma_n}})<8v_3(7n+3).$$

\begin{figure}[h]
	\centering
	\includegraphics[scale=0.23] {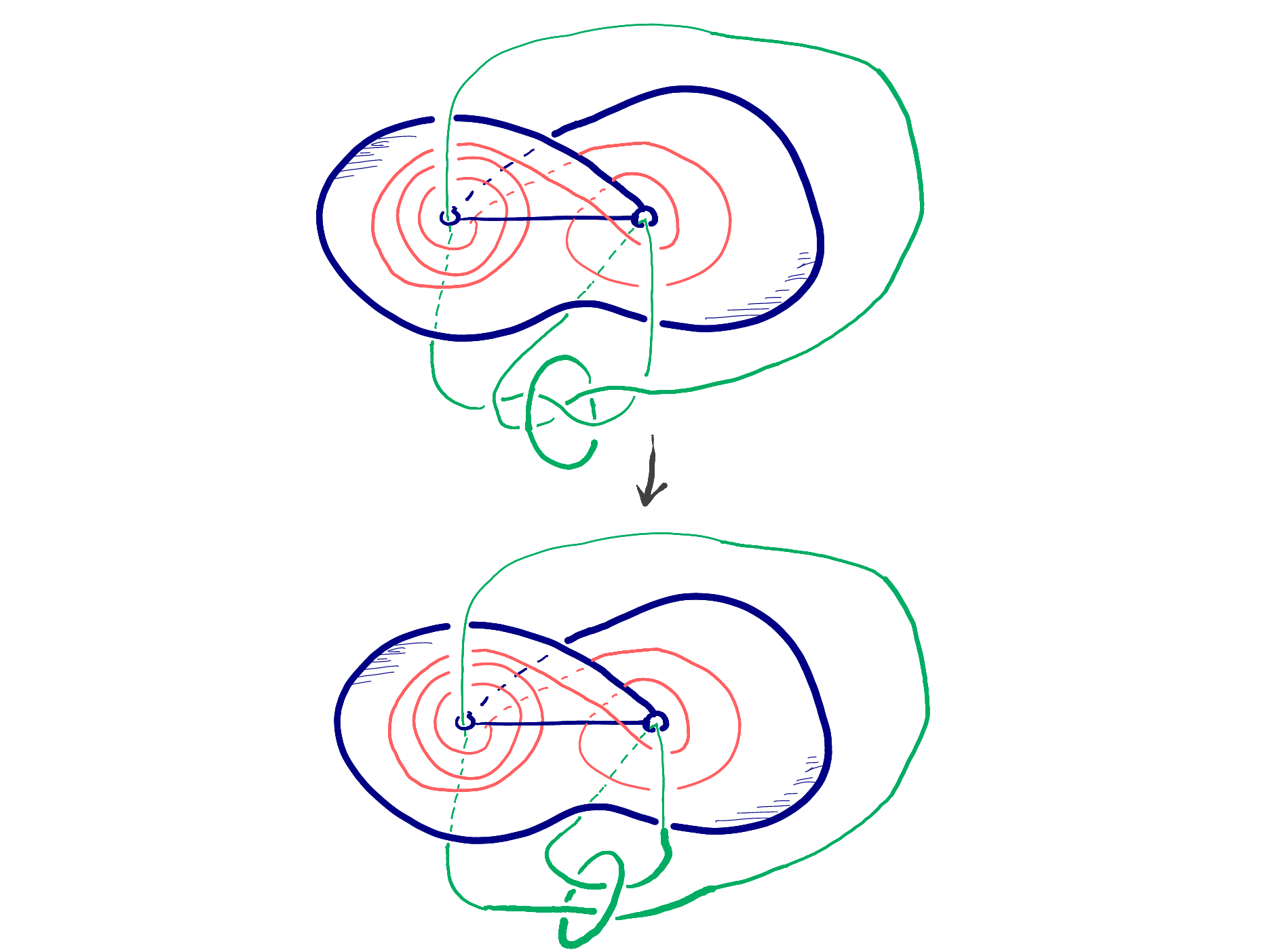}
	\caption{ A $\left(1+\frac{1}{2}\right)$-twist surgery along a thrice-punctured disc bounded by the crossing circle in green. This preserves volume and relates a link complement in $T^1\Sigma_{mod} $ with a link complement in $T^1\Sigma_{0,3}.$ }\label{ada}
\end{figure}

The volume lower bound is a consequence of Theorem \ref{1}.
\vskip .2cm
The proof of this result for any hyperbolic metric on a thrice-puntured sphere, follows from the fact that any pair of hyperbolic metrics on a hyperbolic surface are bi-Lipschitz  (see for example \cite{BPS16}, Lemma 4.1). 
\end{proof}

\vskip .01cm
{ \SMALL Department of Mathematics and Statistics, University of Helsinki. \\Pietari Kalminkatu 5, Helsinki FI 00014}\\
{ \SMALL \textit{E-mail address:} \textbf{jose.rodriguezmigueles@helsinki.fi}}
\end{document}